\documentclass{article}
\usepackage{amsmath,amsfonts,amssymb,amsthm,graphicx}
\usepackage{dsfont}    
\usepackage{enumerate}
\usepackage{color}
\usepackage{multirow}\usepackage{graphicx}
\usepackage{diagbox}
\usepackage{amsfonts}
\usepackage{amsmath}
\usepackage{amssymb}
\usepackage{makecell}
\usepackage{color}
\usepackage{dsfont} 
\numberwithin{equation}{section}
\usepackage[margin=1.25in]{geometry}

\usepackage{pgfplots}
\usepackage{subcaption}
\usepgfplotslibrary{fillbetween}
\usetikzlibrary{patterns}
\usepackage{amssymb}
\usepackage{graphicx}
\usepackage{amsmath}
\usepackage{bm}
\usepackage{tikz}

\usepackage{caption}
  
\usetikzlibrary{decorations.pathreplacing}

\usepackage{tkz-graph}
\usepackage{pgfplots}
\usetikzlibrary{calc,patterns}

\usepackage{subcaption}
\usepackage{url}
\usepackage{fancyhdr}
\usepackage{indentfirst}

\usepackage{enumerate}
\usepackage{dsfont}
\usepackage[british]{babel}
\usepackage[colorlinks=true,citecolor=blue]{hyperref}
\usepackage{amsthm}
\usepackage{color}
\usepackage{natbib}
\usepackage{multirow}
\usepackage{comment}
\usepackage{tikz-cd}
\usepackage{graphicx}
\usepackage{amsfonts}
\usepackage{amsmath}
\usepackage{amssymb}
\usepackage{url}
\usepackage{bbm}
\usepackage{fancyhdr}
\usepackage{indentfirst}
\usepackage{enumerate}
\usepackage{dsfont}
\usepackage[british]{babel}
\usepackage[colorlinks=true,citecolor=blue]{hyperref}
\usepackage{cleveref}
\usepackage{amsthm}
\usepackage{color}
\renewcommand{\cite}{\citet}
\usepackage{comment}

\usepackage{algorithm}
\usepackage[noend]{algpseudocode}  

\renewcommand{\d}{\,\mathrm{d}}
\newcommand{\dd}{\overset{\mathrm{law}}{=}}
\newcommand{\p}{\mathbb{P}}

\newcommand{\var}{\mathrm{Var}}   
\usepackage{xcolor}

\newcommand{\E}{\mathbb{E}}    
\newcommand{\R}{\mathbb{R}}    
\newcommand{\N}{\mathbb{N}}    

\newcommand{\bdpi}{\boldsymbol \pi}

\theoremstyle{plain}
\newtheorem{theorem}{Theorem}

\newtheorem{lemma}[theorem]{Lemma}
\newtheorem{proposition}[theorem]{Proposition}

\theoremstyle{definition}
\newtheorem{definition}{Definition}
\newtheorem{example}{Example}

\newtheorem{fact}{Fact} 
\theoremstyle{remark}
\newtheorem{remark}{Remark} 

\usepackage{tikz}
\usetikzlibrary{arrows.meta,positioning}

\newcommand{\sgn}{\mathrm{sgn}}
\newcommand{\cov}{\mathrm{Cov}}

\newcommand{\bz}{{\mathbf{z}}}
\renewcommand{\P}{\mathcal{P}}

\newcommand{\T}{\mathcal{T}}

\def\d{\mathrm{d}}

\def\laweq{\buildrel \mathrm{law} \over =}

\def\lawis{\buildrel \mathrm{law} \over \sim}

\newcommand{\bone}{ {\mathbbm{1}} }

\renewcommand{\P}{\mathcal{P}}

\usepackage{mathrsfs}
\newcommand{\M}{\mathcal{P}}

\newcommand{\C}{\mathcal{C}}

\newcommand{\X}{\mathfrak{X}}
\newcommand{\Y}{\mathfrak{Y}}

\renewcommand{\ge}{\geqslant}
\renewcommand{\le}{\leqslant}
\renewcommand{\geq}{\geqslant}
\renewcommand{\leq}{\leqslant}
\renewcommand{\epsilon}{\varepsilon}

 \pgfdeclarepatternformonly{south east lines}{\pgfqpoint{-0pt}{-0pt}}{\pgfqpoint{3pt}{3pt}}{\pgfqpoint{3pt}{3pt}}{
    \pgfsetlinewidth{0.4pt}
    \pgfpathmoveto{\pgfqpoint{0pt}{3pt}}
    \pgfpathlineto{\pgfqpoint{3pt}{0pt}}
    \pgfpathmoveto{\pgfqpoint{.2pt}{-.2pt}}
    \pgfpathlineto{\pgfqpoint{-.2pt}{.2pt}}
    \pgfpathmoveto{\pgfqpoint{3.2pt}{2.8pt}}
    \pgfpathlineto{\pgfqpoint{2.8pt}{3.2pt}}
    \pgfusepath{stroke}}

\pgfplotsset{compat=1.7} 
\begin{document}
 
\title{Quadratic-form Optimal Transport}
\author{Ruodu Wang\thanks{Department of Statistics and Actuarial Science, University of Waterloo, Canada. Email: \texttt{wang@uwaterloo.ca}.}\and 
Zhenyuan Zhang\thanks{Department of Mathematics, Stanford University, USA. Email: \texttt{zzy@stanford.edu}.}
 }

\date{\today}

\maketitle

\begin{abstract}
We introduce the framework of quadratic-form optimal transport (QOT), whose transport cost has the form $\iint c\,\mathrm{d}\pi \otimes\mathrm{d}\pi$ for some coupling $\pi$ between two marginals. Interesting examples of quadratic-form transport cost and their optimization include inequality measurement, the variance of a bivariate function, covariance, Kendall's tau, the Gromov--Wasserstein distance, quadratic assignment problems, and quadratic regularization of classic optimal transport. 
 QOT leads to substantially different mathematical structures compared to classic transport problems and many technical challenges.  
 We illustrate the fundamental properties of QOT and provide several cases where explicit solutions are obtained. For a wide class of cost functions, including the rectangular cost functions, the QOT problem is solved by a new coupling called the diamond transport, whose copula is supported on a diamond in the unit square. 
 

\smallskip
\noindent \textbf{MSC classification 2020}: 49Q22; 62H05; 91B70; 62H20
\end{abstract}

\section{Introduction}

Given probability measures $\mu$ on a space $\X$ and $\nu$ on a space $\Y$, a transport plan, also called a  coupling, is a joint distribution on $\X\times \Y$ with marginals $\mu$ and $\nu$. It does not hurt to think of $\X=\Y=\R$ in this section, and $\X$ and $\Y$ will be general Polish spaces in the formal theory.
The set of all such transport plans is denoted by $\Pi(\mu,\nu)$. The classic Kantorovich optimal transport (OT) problem is
\begin{align*}
        \mbox{to minimize}\quad&\int c(x,y)\,\d \pi(x,y)\\
        \mbox{subject to}\quad &\pi\in\Pi(\mu,\nu),
    \end{align*}
where $c:\X\times \Y\to\R$ is a fixed cost function.
This problem can  be written in a probabilistic form:
$$
 \mbox{to minimize}\quad \E[c(X,Y)]~~~~
        \mbox{subject to}~~ X\lawis\mu; ~Y\lawis\nu,
$$
where $X \lawis\mu $ means the distribution of the random variable $X$ is $\mu$. We also call $(X,Y)$ a coupling. The objective $\int c \,\d \pi$ of the OT problem is called the transport cost. 
The OT problem and its numerous extensions have wide applications in various fields including statistics, machine learning, operations research, mathematical finance, and economics. We refer to \cite{Villani:2003,Villani:2009,Santambrogio:2015} for the theory of OT, and \cite{Galichon:2018,peyre2019computational} for applied perspectives.

A classic application of OT in economics and operations research 
concerns the problems of assignment and matching. While the transport cost $\E[c(X,Y)]$ includes many quantities of interest in these problems, such as total production or total cost, 
a larger framework is needed when a notion of equality needs to be optimized, as we illustrate below.

\bigskip
\textbf{Inequality minimization.} 
Suppose that there are two types of non-transferable and indivisible resources A and B with their distributions given by $\mu$ and $\nu$ on $\R$, and a social planner needs to choose a coupling $\pi\in \Pi(\mu,\nu) $  to allocate pairs of resources to (possibly a continuous spectrum of) individuals.
For two individuals with vectors of assigned resources $(x,y)$ and $(x',y')$, their \emph{discrepancy} is defined as a squared weighted sum of differences in each resource, that is, $$(\theta_1|x-x'|+\theta_2 |y-y'|)^2,$$ where the weights $\theta_1,\theta_2\ge 0$ are given.  
If we consider general spaces $\X,\Y$ instead of $\R$, then the discrepancy is $(\theta_1d_{\X} (x,x')+\theta_2 d_{\Y}(y,y'))^2$,
where $d_{\X}$ and $d_{\Y}$ are some distances on $\X$ and $\Y$.

The social planner would like to minimize the average (or total) discrepancy between two randomly selected individuals in the population, that is
\begin{equation}
    \label{eq:motivation1}
  \mbox{to minimize~~~}  \E[(\theta_1 |X-X'|+\theta_2 |Y-Y'|)^2] \mbox{~~~~over $\pi\in \Pi(\mu,\nu)$},
\end{equation}
where  $(X,Y)$ and $(X',Y')$ are independently drawn from the distribution $\pi$.\footnote{This problem has a  Kantorovich formulation. For a practical application with finitely many individuals, one may further require $\pi$ to be induced by a permutation (that is, the Monge formulation), but we will mainly focus on the Kantorovich formulation in our study, which is technically tractable and offers approximations for optimizers in the Monge setting (see Section \ref{sec:diamond}).}

If $\nu$ is degenerate (i.e., all individuals are assigned the same resource B) or $\theta_2=0$ (i.e., discrepancy in resource B does not matter), 
the objective in \eqref{eq:motivation1}
becomes $2\theta^2_1$ times the variance of $X$, which is a natural measure of distributional inequality. 
Hence, the expectation in \eqref{eq:motivation1} can be seen as a generalization of the variance when two different types of quantities are compared simultaneously. 
In addition to the variance, the idea of computing the expected difference between two randomly selected individuals is used to define other classic measures of inequality, such as the Gini deviation  and the Gini coefficient in economics; see Example \ref{ex:gini}.

For a concrete example, suppose that a government agency aims to enhance equality in a population by distributing a menu of economic benefits according to a given distribution $\nu$ (items in the menu cannot be combined or divided). The current wealth level of the population is described by a distribution $\mu$. 
After the financial policy, the wealth and benefits of the population are distributed as $\pi\in \Pi(\mu,\nu)$, determined by the agency. 
Discrepancy between two individuals occurs when either their wealth levels or their benefits differ (or both). 
The problem \eqref{eq:motivation1} is to minimize such discrepancy according to a certain policy $\pi$.

To analyze \eqref{eq:motivation1}, 
one may first look at some commonly encountered  couplings.
A quick observation is that if $\mu=\nu$ and $\pi$ is the comonotone coupling $\pi_{\rm com}$ (formal definition in Section \ref{sec:framework}),
then $Y=X$ and $Y'=X'$ (almost surely), and thus $|X-X'|=|Y-Y'|$. 
{Expanding \eqref{eq:motivation1}, we see that the transport cost (the expectation in \eqref{eq:motivation1}) consists of terms that do not depend on $\pi$, plus $2\theta_1\theta_2\cov(|X-X'|,|Y-Y'|)$.  
Since the distributions of $|X-X'|$ and $|Y-Y'|$
are determined solely by $\mu$ and $\nu$, the covariance and hence the transport cost 
 is maximized when $|X-X'|=|Y-Y'|$.} From there, 
one may then  
 wonder whether the antimonotone coupling  $\pi_{\rm ant}$ minimizes \eqref{eq:motivation1}. However, 
by choosing $\mu=\nu$ as the standard uniform distribution, $\pi_{\rm ant}$  also yields $|X-X'|=|Y-Y'|$, thus also maximizing transport cost. Therefore, intuitively, the optimal coupling must lie somewhere between the most positive coupling $\pi_{\rm com}$  and the most negative coupling $\pi_{\rm ant}$.  It turns out that the optimal coupling, solved in full generality in Section \ref{sec:diamond}, 
 has a special structure that is different from both independence and mixtures of $\pi_{\rm com}$ and $\pi_{\rm ant}$. 
 %

\bigskip

Inspired by the above assignment problem, we propose a new formulation of OT, called the \textit{quadratic-form optimal transport} (QOT).
Given a cost function $c:(\X\times \Y)^2\to \R$, we define the QOT problem as:
\begin{align}
        \mbox{to minimize}\quad&\iint c(x,y,x',y')\,\d \pi(x,y)\,\d \pi (x',y') \label{eq:QOT-cost}\\
        \mbox{subject to}\quad &\pi\in\Pi(\mu,\nu). \notag
    \end{align}
The term ``quadratic-form" reflects that a discrete formulation of the problem \eqref{eq:QOT-cost} can be written into the minimization of a quadratic form (see Appendix \ref{app:QPF}).
This term also distinguishes \eqref{eq:QOT-cost} from optimal transport with the quadratic cost function $c(x,y)=(x-y)^2$, which has been widely studied in the classic OT literature.\footnote{As a side note, the abbreviation QOT was also used for \textit{quadratically regularized optimal transport} in \cite{gonzalez2024sparsity}, \cite{nutz2024quadratically} and \cite{wiesel2024sparsity}; see Example \ref{ex:qf-reg} below.}
Compared to the classic OT problem,  the transport cost  \eqref{eq:QOT-cost} in QOT  is defined as a linear function of $\pi\otimes\pi$ instead of $\pi$ itself, and hence the problem is non-linear. 
Clearly, \eqref{eq:motivation1} is a special case of the QOT problem.

Certain cost functions, such as those that only involve $(x,x')$ or $(y,y')$, have a transport cost determined by the marginals $\mu,\nu$, and they are called \textit{QOT-irrelevant}.

The flexibility of the $4$-variate cost function $c$ allows for a rich spectrum of interesting instances of QOT. We pay special attention to two general sub-classes of cost functions, 
the class of  \emph{type-XX} cost functions,\footnote{In \eqref{eq:typeXX} and \eqref{eq:typeXY}, it should be clear that $f$ maps either $\X^2$ or $\X\times \Y$ to $\R$, $g$ maps either $\Y^2$ or $\X\times \Y$ to $\R$, and $h$ maps   $\R^2$ to $\R$.}
\begin{equation}
\label{eq:typeXX}
c(x,y,x',y')=h(f(x,x'),g(y,y'))\mbox{~~~~for some real-valued bivariate functions $f,g,h$},
\end{equation}
which includes \eqref{eq:motivation1}, 
and
the class of \emph{type-XY} cost functions
\begin{equation}
\label{eq:typeXY}
c(x,y,x',y')=h(f(x,y),g(x',y'))\mbox{~~~~for some real-valued bivariate functions  $f,g,h$}.
\end{equation}
Throughout, equations for the form of cost functions, such as  \eqref{eq:typeXX} and \eqref{eq:typeXY}, are meant to hold for all $(x,y,x',y')$. 
The terms XX and XY reflect the idea that the cost functions in \eqref{eq:typeXX} 
aggregate some costs (e.g., distances) between $x,x'$ and between $y,y'$,
and the cost functions in \eqref{eq:typeXY} 
aggregate costs between $x,y$ and between $x',y'$. It is possible that a non-constant cost function belongs to both types, such as $|x+y+x'+y'|$.  
The two types of cost functions 
lead to very different mathematical structures, which will be explored in this paper. 

Although the QOT framework is much more general than the two classes above, 
 these two types cover many commonly encountered problems. We give a few examples here with $\X=\Y=\R$, with detailed definitions and discussions in Section \ref{sec:exs}. 
The  QOT problem becomes a classic OT problem by choosing  the type-XY cost function
$$c(x,y,x',y')=f(x,y)+g(x',y') \mbox{~~~~for some bivariate functions $f,g$};$$ 
the transport cost 
is 
 the variance of $f(X,Y)$ by choosing   the type-XY cost function $$c(x,y,x',y')=\frac 12 (f(x,y)-f(x',y'))^2\mbox{~~~~for some bivariate function $f$};$$ 
 the transport cost is 
 the covariance of $(X,Y)$ 
 by choosing the  cost function of both types (up to QOT-irrelevant terms) 
 $$c(x,y,x',y')=\frac 12 \underbrace{(x-x')(y-y')}_{\text{type-XX}} = \frac{1}{2}\underbrace{(xy+x'y')}_{\text{type-XY}} - \frac{1}{2} \underbrace{(xy'+x'y)}_{\text{QOT-irrelevant}}; $$
   the transport cost is  
 Kendall's tau  of $(X,Y)$
 by choosing  the type-XX cost function $$c(x,y,x',y')=\sgn(x-x')\,\sgn(y-y');$$
 the optimal transport cost is the $p$-th power of a Gromov--Wasserstein (GW) distance
 by choosing   the type-XX cost function $$c(x,y,x',y')=\big||x-x'|-|y-y'|\big|^p,\qquad p\geq 1;$$  QOT also includes 
the quadratic regularization of classic optimal transport as a special case.    
Moreover, 
a specific Monge formulation of QOT includes
  the quadratic assignment problems  (QAP) of  \cite{Koopmans:1957}
  by choosing  the type-XX cost function  $$c(x,y,x',y')= f(x,x')g(y,y')\mbox{~~~~for some bivariate functions $f,g$}.$$ 
The Koopmans--Beckmann QAP solves  
\begin{align}
    \min_{\sigma\in S_n}\sum_{i=1}^n\sum_{j=1}^na_{ij}b_{\sigma_i\sigma_j},\label{eq:KB}
\end{align}
where $\{a_{ij}\}_{1\leq i,j\leq n},\{b_{ij}\}_{1\leq i,j\leq n}$ are given $n\times n$ matrices and $S_n$ is the set of all permutations of $[n]=\{1,\dots,n\}$. This includes many prominent examples such as the traveling salesman problem. 
Our theory of QOT in the discrete case includes but is much more general than QAP. For instance, if the Monge assumption is relaxed, the optimizer may not even be a (deterministic) assignment; see Example \ref{ex:intro} below.

The non-linearity of QOT induces many difficulties and peculiarities. Since the problem is neither convex nor concave, duality is not generally available, and computational methods are also quite limited. In addition, explicit solutions are rare, while many peculiar examples exist due to the non-linearity. 
We illustrate a simple example below.

\begin{example}\label{ex:intro}
    Let $\mu=\nu$ be the two-point uniform distribution on $\{0,1\}$, that is, Bernoulli($1/2$), and consider the (type-XX) rectangular cost function $c(x,y,x',y')=|x-x'||y-y'|$, which is equivalent to \eqref{eq:motivation1} up to QOT-irrelevant terms. 
    Any $\pi\in\Pi(\mu,\nu)$ can be parameterized by $p=
   2 \pi(\{(1,1)\}) \in [0,1]$, 
   and thus we may write $\Pi(\mu,\nu)=\{\pi_p\}_{p\in[0,1]}$. 
    By direct computation,
\begin{align*}
    \iint c\,\d\pi_p\otimes\d\pi_p&=
    2 \pi_p\otimes \pi_p (\{(0,0,1,1), (0,1,1,0)\} ) 
 =2\left(  \frac{p^2}{4} +  \frac{(1-p)^2}{4}\right) 
    =p^2-p+\frac{1}{2}.
\end{align*}
Therefore, the transport cost $\iint c\,\d\pi_p\otimes\d\pi_p$ is a quadratic function in $p$ that is uniquely minimized by $p=1/2$. In other words, the independent coupling is the unique minimizer. On the contrary, it is well-known that for classic OT, if $\mu,\nu$ are both uniformly distributed on the same number of points, a bijective optimizer exists, which follows from Birkhoff's theorem (\cite{Birkhoff:1946}).  
\end{example}

\begin{figure}[t]
\centering
\begin{subfigure}{0.19\textwidth}
    \includegraphics[width=1\linewidth]{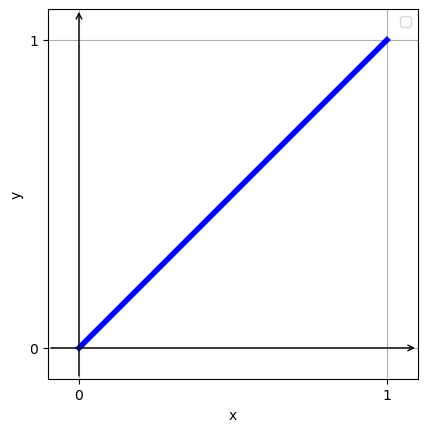}
    \caption{$\pi_{\rm com}$}
\end{subfigure}
\begin{subfigure}{0.19\textwidth}
\centering \includegraphics[width=1\linewidth]{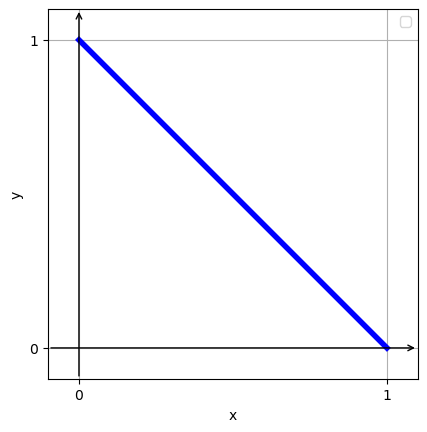}
\caption{$\pi_{\rm ant}$}
\end{subfigure} 
\begin{subfigure}{0.19\textwidth}
    \includegraphics[width=1\linewidth]{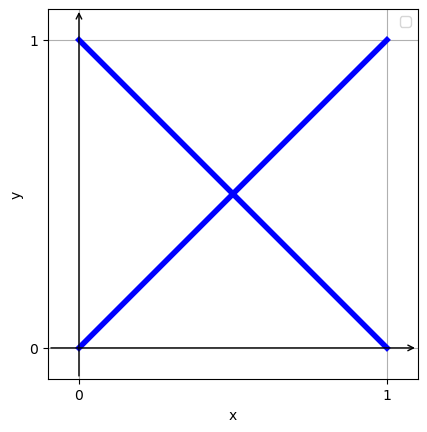}
    \caption{$\pi^{\lambda}_{\rm x}$}
\end{subfigure}
\begin{subfigure}{0.19\textwidth}
\centering \includegraphics[width=1\linewidth]{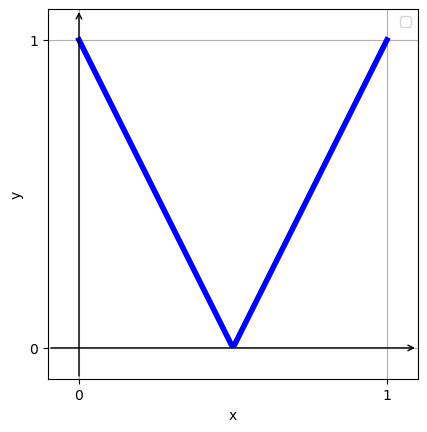}
\caption{$\pi_{\rm v}$}
\end{subfigure}
\begin{subfigure}{0.19\textwidth}
\centering \includegraphics[width=1\linewidth]{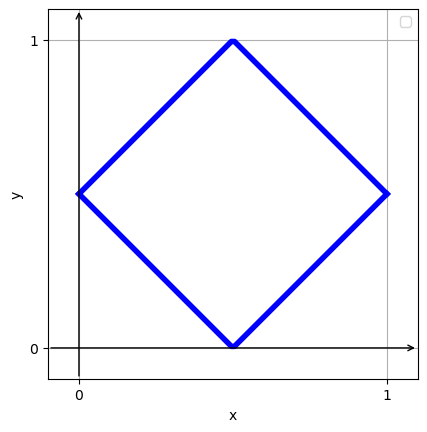}
\caption{$\pi_{\rm dia}$ }
\end{subfigure}
\caption{Illustration of the support of
the comonotone transport $\pi_{\rm com}$,
the antimonotone transport $\pi_{\rm ant}$,
the X-transport $\pi_{\rm x}^\lambda$ with $\lambda \in(0,1)$,
the V-transport $\pi_{\rm v}$,
and the diamond transport $\pi_{\rm dia}$,
which appear as QOT minimizers 
with marginals normalized to uniform distributions on $[0,1]$; for their corresponding cost functions and marginals, see Table \ref{tab:example} below. Blue lines indicate the support of the transport plans. The transport plans in (c), (d), and (e) appear new compared to classic OT. Precise definitions are given in Section \ref{sec:framework} and Definitions \ref{def:wedge} and \ref{def:diamond}.}
\label{fig:transport} 
\end{figure}

Given the richness of possible special cases and applications 
and the mathematical novelty of the new framework, we dedicate this paper to a systematic study of QOT.  
Our main contributions are summarized below.

\begin{table}[t]
\centering\renewcommand{\arraystretch}{1.5} 
\small
\begin{tabular}{c|c|c|c} 
{Cost function $c$} & {Marginals $\mu,\nu$} & {Minimizer} & {Location} \\ \hline 
\makecell{$f(x,y)g(x',y')$\\ $f,g$ quadratic }& general      & \rule{0pt}{20pt} \makecell{  $\pi_{\rm x }^\lambda= \lambda \pi_{\rm com}+(1-\lambda)\pi_{\rm ant}$\\ for some $\lambda\in [0,1]$ }     & Proposition \ref{ex:quadratic cost}      \\ \hline
\rule{0pt}{20pt}
\makecell{{$(\alpha_1f(x,y)+C_1)(\alpha_2f(x',y')+C_2)$}\\
$f$ submodular, $C_1,C_2,\alpha_1,\alpha_2\in \R$}    & general   & \makecell{  $\pi_{\rm x}^\lambda= \lambda \pi_{\rm com}+(1-\lambda)\pi_{\rm ant}$\\ for some $\lambda\in [0,1]$ }    & Proposition \ref{ex:quadratic cost-f}    \\ 
\hline
\rule{0pt}{25pt}
\makecell{$(x',y')\mapsto c(x,y,x',y')$
\\ and $(x,y)\mapsto c(x,y,x',y')$\\
both submodular}    & general   & $\pi_{\rm com}$    & Theorem \ref{prop:comonotone}    \\ \hline
\rule{0pt}{20pt}
\makecell{$h(|x-x'|,|y-y'|)$\\ $h$ submodular}      &\makecell{$\nu=\mu\circ \ell^{-1}$\\ $\ell$ linear map} & \makecell{$\pi_{\rm com}$ ($\ell$ increasing) \\ $\pi_{\rm ant}$ ($\ell$ decreasing)}
& Theorem \ref{prop:cocounter}      \\ \hline 
\rule{0pt}{30pt}
\makecell{$f(|x-x'|)g(y,y')$\\  $f$ nonnegative increasing\\ $g$ increasing supermodular\\ and some regularity conditions}     & \makecell{$\mu $   is uniform\\ on an interval}   & $\pi_{\rm v}$   & Theorem \ref{prop:wedge}     \\ \hline
\rule{0pt}{25pt} \makecell{$|(x-x')(y-y')|$\\or $(\theta_1|x-x'|+\theta_2|y-y'|)^2$\\with $\theta_1,\theta_2\geq 0$}      & general   & $\pi_{\rm dia}$   &   Theorem \ref{prop:pidia}   \\ \hline

\rule{0pt}{25pt} \makecell{$\phi((x-x')^2)\phi((y-y')^2)$\\ $\phi$ completely monotone\\ $\phi'(u)+2u\phi''(u)\leq 0$}      & $\mu,\nu$ symmetric     & $\pi_{\rm dia}$      &  Theorem \ref{thm:diamond}  \\ \hline
\rule{0pt}{15pt}
$|(x-x')(y-y')|^q,~q\in(1,2]$      & $\mu,\nu$ symmetric    & $\pi_{\rm dia}$   &   Theorem  \ref{thm:p-cost}  \\ \hline
\end{tabular}\\~
\caption{A selected list of explicitly solved examples of QOT problems on the real line. The minimizers may not be unique. Certain moment assumptions on the marginals are omitted (compactness of support is sufficient).
For definitions of the couplings, see Section \ref{sec:framework} and Definitions \ref{def:wedge} and \ref{def:diamond}. 
In   Examples 
\ref{ex:sep}--\ref{ex:th-sub},
many more explicit cost functions, some of which belong to the above general classes, are presented. The conclusions remain the same if QOT-irrelevant terms like $w_1(x,x')+w_2(y,y')+w_3(x,y') + w_4(x',y)$ are added to the cost function $c$ (Fact \ref{rem:marginals}).
} 
\label{tab:example}
\end{table}

In Section \ref{sec:framework}, we formally present the framework of QOT on general Polish spaces. 
 {Section \ref{sec:exs} collects many relevant examples of QOT in optimization, economics, computational OT, and statistics.}

 Section \ref{sec:general} provides general results on QOT, including properties of the optimizers and general lower bounds of the QOT cost, some of which extend known results on the GW distance. 
Section \ref{sec:explicit} examines several explicit solvable cases. We show that the comonotone  coupling $\pi_{\rm com}$ and the antimonotone coupling $\pi_{\rm ant}$, as well as their mixtures, form solutions to many classes of cost functions, including quadratic costs, jointly submodular/supermodular costs, and Gromov--Wasserstein-type costs.

Due to the fundamental differences from the classic OT setting, many new optimal transport plans emerge besides the comonotone and antimonotone ones. Figure \ref{fig:transport} illustrates the support of some  QOT minimizers. The most interesting one is arguably the diamond transport $\pi_{\rm dia}$, for the following reasons: first, it does not appear in other contexts of OT or QAP; second, it is perfectly symmetric but is not Monge; third, it serves as a universal minimizer of a large class of QOT problems with some assumptions on the marginals. A simple example with a diamond minimizer is given by the rectangular cost function $c(x,y,x',y')=|(x-x')(y-y')| $ in Example \ref{ex:intro}, equivalent to the one in \eqref{eq:motivation1}. 
Section \ref{sec:diamond} focuses on the   diamond transport.

 Section \ref{sec:concl} concludes with a few open questions. The appendices contain further discussions, additional results, and omitted proofs. Appendix \ref{app:QPF} gives a quadratic programming formulation of QOT. Appendix \ref{sec:no ind} shows that the independent coupling is rarely, but possibly, an optimizer of the QOT, and gives some interesting examples. In Appendix \ref{sec:dispersion}, we consider the class of linear-exponential distance cost functions of the form
$c(x,y,x',y')=|y-y'|e^{-\gamma |x-x'|},~\gamma>0$.
This class of cost functions
are minimized by the comonotone coupling, but its maximizers have interesting limiting behavior as $\gamma $ goes to $0$ or $\infty$, such as the diamond and independent couplings, in some special senses. In particular, the limitting case $\gamma\to\infty$ is connected to recently studied measures of association (\cite{chatterjee2021new,deb2020measuring}).
Appendix \ref{sec:EC-disc} contains a more detailed discussion of the open questions from Section \ref{sec:concl}. Appendices \ref{app:sec3}--\ref{app:sec5} collect omitted proofs of the results from Sections \ref{sec:general}--\ref{sec:diamond}.
 
Before moving on to the formal analysis, we summarize in Table \ref{tab:example} the QOT problems with known explicit optimizers obtained in this paper.

\section{Framework}\label{sec:framework}

As in the Introduction, 
let $\mathfrak X$
and $\mathfrak Y$ be two Polish spaces,    $\mu$ and $\nu$
be two probability measures on $\mathfrak X$
and $\mathfrak Y$ respectively, and $\Pi(\mu,\nu)$ be the set of all distributions on $\X\times \Y$ with marginals $\mu,\nu$.  In many explicit results and examples, we will take $\X=\Y=\R$, but we also present some results on more general spaces. 
For a function $c:(\X\times\Y)^2\to \R$, 
called a cost function,
and a coupling $\pi\in \Pi(\mu,\nu)$, 
define the 
\emph{quadratic-form transport cost} as\footnote{Throughout, we tacitly assume suitable measurability of the cost function $c$ so that \eqref{eq:QOT cost} is meaningful.} 
\begin{align}\label{eq:QOT cost}
\iint c\,\d\pi\otimes\d\pi=\iint c(x,y,x',y')\,\d \pi(x,y)\,\d \pi (x',y')
\end{align} 
and the (Kantorovich) QOT problem is to minimize (and occasionally, to maximize) this transport cost over all $\pi \in \Pi(\mu,\nu)$ such that the integral \eqref{eq:QOT cost} is well-defined (taking possibly infinite values). We omit ``Kantorovich'' in the sequel. 
The probabilistic formulation of the quadratic-form transport cost \eqref{eq:QOT cost} is 
$
\E[{c}(
\mathbf Z,\mathbf Z')]
$,
where $\mathbf Z,\mathbf Z'\lawis\pi$ are iid.
A quadratic program formulation of discrete QOT is presented in Appendix \ref{app:QPF}.

\begin{fact}\label{rem:marginals}
\sloppy    The QOT problem remains equivalent (that is, with the same set of minimizers) if  QOT-irrelevant terms are added to the cost function.
For instance, the cost functions $c$ and $$(x,y,x',y')\mapsto c(x,y,x',y')+w_1(x,x')+w_2(y,y') +w_3(x,y') + w_4(x',y)$$ lead  to equivalent QOT problems. 
All results in this paper automatically hold when QOT-irrelevant terms are added to the cost functions.
\end{fact}

In certain applications, one may restrict to the Monge setting, where the set of couplings is induced by functions. Denote by $\T(\mu,\nu)$ the set of measurable maps $T:\X\to\Y$ satisfying $\mu\circ T^{-1}=\nu$, also known as the set of transport maps (or Monge maps) from $\mu$ to $\nu$. The \textit{Monge QOT} problem is to minimize 
$$\iint c(x,T(x),x',T(x'))\,\d\mu(x)\,\d\mu(x'),$$
over the set $T\in\T(\mu,\nu)$.

The QOT problem can be realized as a variation of the multi-marginal OT problem under independence and marginal constraints. Consider the set $\Pi(\mu,\nu,\mu,\nu)$ of all probability measures on $(\X\times\Y)^2$ with the four marginals given respectively by $\mu,\,\nu,\,\mu,\,\nu$. The \textit{multi-marginal optimal transport} problem minimizes the transport cost
$$\int c(x,y,x',y')\,\d\tilde  \pi(x,y,x',y')$$
over $\tilde \pi\in \Pi(\mu,\nu,\mu,\nu)$; see \cite{pass2015multi,pass2024general} for surveys. Let $\Pi_{\mathrm{ind}}(\mu,\nu,\mu,\nu)$ be the couplings $(X,Y,X',Y')$ of $\mu,\,\nu,\,\mu,\,\nu$ such that $(X,Y)$ and $(X',Y')$ are independent and have the same distribution. We then arrive at the equivalence 
$$\inf_{\tilde \pi\in \Pi_{\mathrm{ind}}(\mu,\nu,\mu,\nu)}\int c(x,y,x',y')\,\d \tilde \pi(x,y,x',y')=\inf_{\pi\in\Pi(\mu,\nu)}\iint c(x,y,x',y')\,\d \pi(x,y)\,\d \pi (x',y').$$

In the rest of this section, we recall some fundamental results in classic OT and set up the necessary notation. 
Recall that a function $f:\R^2\to\R$ is called \textit{submodular} if for any $x<x'$ and $y<y'$,
\begin{align}
    f(x,y)+f(x',y')\leq f(x,y')+f(x',y),\label{eq:sub}
\end{align}
and $f$ is called \textit{supermodular} if for any $x<x'$ and $y<y'$,
\begin{align}
    f(x,y)+f(x',y')\geq f(x,y')+f(x',y).\label{eq:sup}
\end{align} 
In case  the inequalities in \eqref{eq:sub} and \eqref{eq:sup} are strict, we say $f$ is strictly submodular (or supermodular). Assuming $f\in C^2(\R^2)$,  the cross partial derivative $f_{xy}$ is nonnegative (resp.~nonpositive) if and only if $f$ is supermodular (resp.~submodular).

Two classic couplings are fundamental to classic OT on $\mathfrak X\times \mathfrak Y=\R^2$, which we define below.
 Denote by $\M(\R)$ the set of probability measures on $\R$. 
For a probability measure $\mu\in\M(\R)$, 
let $Q_\mu $ be the left quantile function of $\mu$, that is, $Q_\mu(t)=\inf \{x\in \R: \mu((-\infty,x]) \ge t\}$ for $t\in [0,1]$ with $\inf\varnothing =\infty$. 
A coupling $(X,Y)$ with marginals $\mu,\nu\in\M(\R)$, or its joint distribution, is \textit{comonotone} if 
$(X,Y)\laweq (Q_\mu(U),Q_{\nu}(U))$,
where $U$ is uniformly distributed on $[0,1]$; it
is \textit{antimonotone}  if 
$(X,Y)\laweq (Q_\mu(U),Q_{\nu}(1-U))$.
It is well known that these two couplings either maximize or minimize classic OT problems when the cost function is submodular or supermodular (e.g., Theorem 2.9 of \cite{Santambrogio:2015}). We let $\pi_{\mathrm{com}}\in\Pi(\mu,\nu)$ denote the comonotone coupling, $\pi_{\mathrm{ant}}\in\Pi(\mu,\nu)$ denote the antimonotone coupling, and $\pi_{\mathrm{ind}}=\mu\otimes\nu\in\Pi(\mu,\nu)$ denote the independent coupling.
Couplings such as $\pi_{\rm com}$, $\pi_{\rm ant}$, and $\pi_{\rm ind}$ depend on the marginals $\mu,\nu$, which should be clear from context.
 In addition, for $\lambda\in[0,1]$,  let $\pi_{\rm x}^\lambda=\lambda\pi_{\rm com}+(1-\lambda)\pi_{\rm ant}\in\Pi(\mu,\nu)$; 
 the coupling 
 $\pi_{\rm x}^\lambda$ for $\lambda \in (0,1)$ is called an \emph{X-transport} because its support has an X-shape;
 see Figure \ref{fig:transport}, panel (c).

Some further notation and terminologies will be used throughout the paper. 
We say that a measure $\mu\in\M(\R)$ is \textit{symmetric} if there exists $m\in \R$ such that  
$\mu(A)=\mu(m-A)$ for all Borel sets $A\subseteq\R$. 
This means $X\lawis \mu  \iff  m-X\lawis \mu$. 
Otherwise, we say $\mu$ is \textit{asymmetric}. For a probability measure $\mu$ on $\R^d$ with $d\in \N$, we denote by
  $F_{\mu}$ the cdf of $\mu$. For $p\geq 1$, we let $\P_p(\R)$ denote the set of probability measures $\mu\in\M(\R)$ with a finite $p$-th absolute moment, i.e., $\int_\R |x|^p\d\mu(x)<\infty$. For $a<b$, we denote by $\mathrm{U}(a,b)$ the uniform distribution on $[a,b]$. We also write $\mathrm{U}=\mathrm{U}(0,1)$.
Denote by $\R_+$ the set of nonnegative real numbers.

\section{Examples of QOT}\label{sec:exs}

In this section, we discuss several examples of QOT that appear in different fields {and how they connect to the results in this work}.

\begin{example}[Sum of bivariate functions]
\label{ex:sep}
Suppose that the cost function  $c$ can be written as the sum of several bivariate functions, that is, 
$$
c(x,y,x',y')= f(x,y) +g(x',y') + w_1(x,x')+w_2(y,y') +w_3(x,y') + w_4(x',y).
$$
By  Fact \ref{rem:marginals}, 
the QOT problem with the above cost is equivalent to   
 the type-XY cost function  $  g(x,y)+h(x',y')$.
In this case, 
 $$\int \left(f(x,y)+ g(x',y')\right ) \d \pi (x,y) \d \pi (x',y')=\int (f+g)\,\d \pi . $$
  In other words, the QOT problem reduces to a classic OT problem with cost function $f+g$.
  In particular, if $f+g$ is submodular, then the comonotone coupling is a minimizer. 
  On the other hand, if the cost function $c $   has a multiplicative form
  $c(x,y,x',y')=f(x,y)g(x',y')$, 
  then  $\iint c\,\d\pi\otimes\d\pi   = (\int  f\, \d \pi )(\int g\, \d \pi), $
which is the product of two transport costs in classic OT. If $f$ and $g$
 are both submodular and nonnegative, a minimizer is the comonotone coupling. 
 \end{example}

\begin{example}[Variance minimization with given marginals]\label{ex:var}

 Let $f: \X\times\Y \to \R$ be a measurable function. 
 Suppose that the goal is to minimize the variance of $f(X,Y)$ subject to $X\lawis \mu$ and $Y\lawis \nu$. 
 This problem is QOT with the type-XY nonnegative cost function  $c$ given by 
 $c(x,y,x',y')=(f(x,y)-f(x',y'))^2/2$ because, for $(X,Y)\lawis \pi$,
\begin{align}
    \label{eq:var-min} 
    \begin{aligned}
  \iint c\,\d\pi \otimes \d\pi &=\begin{cases}\E[f(X,Y)^2]- \E[f(X,Y)]^2 & \text{ if }\E[f(X,Y)^2]<\infty 
\\\infty & \mbox{ otherwise}
\end{cases} 
\\ &= \var(f(X,Y)). 
    \end{aligned}
\end{align} 
This QOT problem is well-posed even when $f(X,Y)$ does not have a finite variance for some coupling $\pi$.
The minimization of \eqref{eq:var-min} is not a classic OT problem, because transport costs in classic OT are linear in $\pi$, whereas \eqref{eq:var-min}  is not.  

\end{example}
\begin{example}[Covariance]
\label{ex:cov}
Assume that $\mu,\nu\in\P_2(\R)$. Consider the QOT problem with the type-XX (and also type-XY, up to QOT-irrelevant terms) cost function $c$ given by 
$$
c(x,y,x',y')= \frac 12 (x-x')(y-y') = \frac{1}{2}(xy+x'y')-\frac 12(xy'+x'y).
$$
For $(X,Y)\lawis \pi$,  we can verify 
\begin{align*}
\iint c  \, \d \pi \otimes \d \pi&=
\frac12 \left(\E[XY]+\E[X'Y'] -\E[XY']-\E[X'Y]\right)
\\&=\E[XY] - \E[X]\E[Y]=\cov(X,Y).
\end{align*}
Therefore, the transport cost is the covariance of $(X,Y)$. It is well-known that the unique minimizer of covariance is the antimonotone coupling and the unique maximizer is the comonotone coupling, which is also a consequence of Theorem \ref{prop:comonotone}.

\end{example}
\begin{example}[Kendall's tau]\label{ex:kendall tau}
    Kendall's tau, also called Kendall's rank correlation coefficient, is one of the most popular measures of bivariate rank correlation, widely used in statistics and stochastic modeling; see  e.g., \citet[Chapter 5]{nelsen2006introduction}  and \citet[Chapter 7]{mcneil2015quantitative}. For a random vector $(X,Y)$ taking values in $\R^2$, its Kendall's tau is defined as
    $$
    \tau = \E[\sgn((X-X')(Y-Y'))],    
    $$
    where $(X',Y')$ is an independent copy of $(X,Y)$. Intuitively, it equals the probability of concordance 
    minus  that  of discordance between $(X,Y)$
    and $(X'    ,Y')$.  
    Clearly, $\tau $ of $(X,Y)\lawis \pi$ can be written as the quadratic-form transport cost
    with the type-XX cost function $c(x,y,x',y')=\sgn(x-x')\,\sgn (y-y')$.
    For given marginals $\mu,\nu$, it is well-known that $\tau(\pi)$ over $\pi\in \Pi(\mu,\nu)$ is maximized by the comonotone coupling  with maximum value $1$
    and minimized by the antimonotone coupling with minimum value $-1$ (this can also be checked by Theorem \ref{prop:comonotone}; see Example \ref{ex:th-sub}).
  Another equivalent formulation is
    $
     \tau(\pi) = 4 \int  \pi \d \pi  -1,
    $
    from which the quadratic form in $\pi$ is visible.

\end{example}

\begin{example}[Gini  deviation and Gini coefficient]\label{ex:gini}
Let $L^1$ be the set of integrable random variables and $L^1_+=\{Z\in L^1: Z\ge 0;~ \E[Z]>0\}.$
Define the mappings $\mathrm{GD}$ and $\mathrm{GC}$ on $L^1_+$ by  
$$
\mathrm{GD} (Z)=\frac 12 \E[|Z-Z'|]
\mbox{~~~and~~~} \mathrm{GC} (Z)=\frac{\mathrm{GD}(Z)}{\E[Z]} = \frac{\E[|Z-Z'|]}{\E[Z+Z']},
$$
where $Z'$ is an independent copy of $Z$.
The value  $\mathrm{GD} (Z)$ is called the Gini deviation of $Z$, 
and  $\mathrm{GC} (Z)$ is called the Gini coefficient of $Z$, both of which are commonly used as measures of distributional variability or inequality in  economics and risk management; see e.g., \cite{gastwirth1971general} and \cite{furman2017gini}.  
Similarly to the variance in \eqref{eq:var-min}, minimization of 
 $\mathrm{GD}(f(X,Y))$ for some measurable $f:\R^2\to \R$  
 over $X\lawis \mu$ and $Y\lawis \nu $
 can be written as
the QOT problem with the type-XY cost function $|f(x,y)- f(x',y')|$. 
Moreover, the minimization of 
 {$\mathrm{GC}( X+Y)$} can be written as 
\begin{equation}
    \label{eq:GC}
\min_{\pi\in \Pi(\mu,\nu)} \frac{\iint |x+y- x'-y'|  \,\d \pi(x,y)\,\d \pi (x',y') }
{ 2 (\int x \d \mu(x) + \int y \d \nu (y)) },
\end{equation} 
which is equivalent to a QOT problem with cost function   $c(x,y,x',y')=|x+y- x'-y'| $, noting that the denominator of \eqref{eq:GC} does not involve $\pi$. {A minimizer of \eqref{eq:GC}  is given by the antimonotone coupling $\pi_{\mathrm{ant}}$; see Example \ref{ex:th-sub}(v) below.}
This cost function is both type-XX and  type-XY.



\end{example}

\begin{example}[Gromov--Wasserstein (GW) distance]
\label{ex:GW}
A special case of  QOT is the GW distance, a measure of the distance (or similarity) between two metric measure spaces  introduced and studied by \cite{memoli2007use,memoli2011gromov}.  
Suppose that $(\X,d_\X,\mu)$ and $(\Y,d_\Y,\nu)$ are metric measure spaces. For $p,q\geq 1$, the $(p,q)$-GW distance is defined as
\begin{align}
    \mathrm{GW}_{p,q}(\X,\Y):=\bigg(\inf_{\pi\in\Pi(\mu,\nu)}\iint |d_\X(x,x')^q-d_\Y(y,y')^q|^p \d \pi(x,y)\,\d \pi (x',y')\bigg)^{1/p}.\label{eq:GWpq}
\end{align}
This is an increasing transform of the minimum transport cost of QOT with a type-XX cost function. 
    The GW distance satisfies the triangle inequality and defines a pseudo-metric on metric measure spaces (and a metric on isomorphism classes of metric measure spaces; see \citet[Corollary 9.3]{sturm2023space}).
    The GW distance is a widely used technique in data science, machine learning, computer vision, and computer graphics to align heterogeneous data sets or images (\cite{memoli2011gromov,peyre2019computational}). However, in general, solving for the GW distance is a challenging task. 
In Theorems \ref{prop:pidia} and \ref{thm:p-cost}, we explicitly characterize the \textit{maximizers} of the transport cost that appears in \eqref{eq:GWpq} for $p=2$, $q\in[1,2]$, $\X=\Y=\R$ equipped with the Euclidean distance, 
    and symmetric marginals, where the symmetry is not needed for $q=1$. The GW literature also incorporates some earlier ideas discussed in this paper, such as existence of minimizers and Monge--Kantorovich equivalence (Proposition \ref{prop:existence uniqueness}), lower bounds on the transport cost (Propositions \ref{prop:lb from conditioning} and \ref{prop:lb from multimarginal}), and connections to QAP. These ideas also manifest in studies of extensions of the GW distance (\cite{arya2024gromov,bauer2024z,chowdhury2019gromov,memoli2011spectral,memoli2012some,memoli2023ultrametric}), sometimes leading to closed-form solutions. 
\end{example}

\begin{example}[Quadratic Assignment Problem (QAP)]\label{ex:QAP}
If the probability measures $\mu,\nu$ are each uniformly distributed on $N$ points, the Monge QOT problem reduces to QAP,  
 which was first introduced by \cite{Koopmans:1957} as a model for the allocation problem of indivisible economic activities. This connection is well known in the GW literature; see, for instance, Remark 4.6 of \cite{memoli2011gromov}. 
 Recall the Koopmans--Beckmann problem from \eqref{eq:KB}. The work of \cite{lawler1963quadratic} proposed a generalization of the Koopmans--Beckmann QAP, where one solves
$$\min_{\sigma\in S_n}\sum_{i=1}^n\sum_{j=1}^nd_{ij\sigma_i\sigma_j},$$
where $D=\{d_{ijk\ell}\}_{1\leq i,j,k,\ell\leq n}$ is a given $4$-index cost array. 
Since any Monge transport map between $\mu$ and $\nu$ matches the $N$ elements bijectively, the Lawler QAP coincides with Monge QOT with discrete uniform marginals, and the Koopmans--Beckmann QAP further constrains that the cost function is of the form  $c(x,y,x',y')=c_1(x,x')c_2(y,y')$, thus type-XX. 
The QAP class includes many well-known combinatorial optimization problems, such as the traveling salesman problem (Section 7.1.2 of \cite{burkard1998quadratic}) and the campus planning model (\cite{dickey1972campus}). However, in general, even approximating QAP is NP-hard (\cite{loiola2007survey,queyranne1986performance}). For instance, QAP of size $n>30$ cannot be solved in a reasonable amount of time.
We refer to \cite{burkard2012assignment} and \cite{cela2013quadratic} for comprehensive surveys on QAP and various extensions of the problem.  

\end{example}

\begin{example}[Quadratic regularization of discrete optimal transport]\label{ex:qf-reg}
    Regularization is a modern technique in OT that facilitates computation by introducing strong convexity to the linear OT problem. The most prevalent choice is arguably the entropic regularized OT (EOT), which enables the Sinkhorn's algorithm and generates smoothness properties of the solution (\cite{nutz2021introduction}).
    An alternate of EOT is given by the following quadratically regularized OT: 
    \begin{align}
        \mbox{to minimize}\quad&\int c(x,y)\,\d \pi(x,y)+\frac{\varepsilon}{2}\int\Big(\frac{\d\pi}{\d \pi_{\mathrm{ind}}}(x,y)\Big)^2\d \pi_{\mathrm{ind}}(x,y)\label{eq:qreg}\\
        \mbox{subject to}\quad &\pi\in\Pi(\mu,\nu),\nonumber
    \end{align}
    where by convention, the objective value is $\infty$ if $\pi\not\ll \pi_{\mathrm{ind}}$.
The quadratically regularized OT was introduced by \cite{blondel2018smooth} and \cite{essid2018quadratically} in the discrete setting, and rigorously studied by \cite{lorenz2021quadratically} in the continuous case. The authors highlighted that quadratically regularized OT gives rise to sparse couplings, a desirable property when the OT itself is of interest. 
Another advantage of quadratically regularized OT over EOT is the allowance of small regularization parameters, as the computation for EOT is difficult for a small $\varepsilon$ (\cite{nutz2024quadratically}). 
In the discrete setting, suppose that $\mu$ has mass $\{p_i\}$ on points $\{x_i\}$ and $\nu$ has mass $\{q_i\}$ on points $\{y_i\}$. Denote also by $\pi_{ij}$ the mass of $\pi$ on $(x_i,y_j)$. The regularization term of \eqref{eq:qreg} can be written as
\begin{align*}
    \int\Big(\frac{\d\pi}{\d \pi_{\mathrm{ind}}}(x,y)\Big)^2\d \pi_{\mathrm{ind}}(x,y)=\sum_i\sum_j \frac{\pi_{ij}^2}{p_iq_j}
    &=\sum_{i,j}\sum_{k,\ell} \frac{1}{p_iq_j}\bone_{\{i=k,j=\ell\}}\pi_{ij}\pi_{k\ell}\\&=\iint \frac{\bone_{\{x=x',y=y'\}}}{\mu(\{x\})\nu(\{y\})}\d\pi(x,y)\,\d\pi(x',y'),
\end{align*}
which is a transport cost in QOT (type-XX) with the unique minimizer given by $\pi_{\rm ind}$, a consequence of Jensen's inequality. Since the classic OT is also a special case of QOT (Example \ref{ex:sep}), the quadratically regularized OT in the discrete case belongs to the QOT class.
\end{example}

\section{General properties}
\label{sec:general}

{We now proceed to study general properties of QOT.}
 The results presented in this section hold for probability measures on general Polish spaces, except for the stability of QOT (Proposition \ref{prop:stability}), where we require $\X=\Y=\R$.

Convexity is an essential issue in classic OT theory, giving rise to many useful techniques such as duality and $c$-cyclical monotonicity. 
We start with a simple result that gives a sufficient condition for QOT to be convex. For a non-empty set $S$, we say that a symmetric {kernel} $\phi:S\times S\to \R$ is {a \textit{positive definite kernel}} if for any $n\in\N,~s_1,\dots,s_n\in S$, and $c_1,\dots,c_n\in\R$, it holds
$$\sum_{i=1}^n\sum_{j=1}^n c_ic_j\phi(s_i,s_j)\geq 0,$$
which is a generalization of positive semi-definite matrices. {A closely related concept is that of \emph{positive definite functions}, that is, functions $\varphi:\R\to\R$ 
satisfying  that $(x,y)\mapsto \varphi(x-y)$ is a positive definite kernel on $\R^2$.}
\begin{proposition}\label{prop:convex?}
    If the cost function $c$ satisfies $c(x,y,x',y')=\phi((x,y),(x',y'))$ for some $\phi:(\X\times\Y)^2\to\R$ that is bounded, continuous, and positive definite, then QOT is a convex optimization problem. In other words, the quadratic-form transport cost \eqref{eq:QOT cost} is a convex function of $\pi$.
\end{proposition}

\begin{example}
    Denote by $\|\cdot\|$ the Euclidean norm. For $\bz,\bz'\in\R^d$, the kernels $(\bz,\bz')\mapsto e^{-\alpha\|\bz-\bz'\|}$ and $(\bz,\bz')\mapsto e^{-\alpha\|\bz-\bz'\|^2}$ are both positive definite for $\alpha>0$. Therefore, the QOT problem with the type-XX cost function $e^{-\alpha\|(x,y)-(x',y')\|}$ or $e^{-\alpha\|(x,y)-(x',y')\|^2}$ is convex. We will see more examples in Example \ref{ex:dia}, where we also show that QOT with cost function $e^{-\alpha\|(x,y)-(x',y')\|^2},\,\alpha\in(0,1/2]$ is minimized by the diamond transport.
\end{example}
 
Despite the above result, the majority of QOT problems are not convex in $\pi$. 
This non-convex structure of QOT prohibits the use of classic tools such as duality.\footnote{Notable exceptions include \citet[Theorem 4.2.5]{vayer2020contribution} and \cite[Theorem 1]{zhang2024gromov}, and the latter result also analyzes sample complexity of the $(2,2)$-GW distance; see Example \ref{ex:GW}.} 
In the rest of this section, we study the fundamental properties of QOT, which may not be convex. Specifically, we show that under certain assumptions, minimizers exist, the Monge optimal  transport cost is equivalent to the Kantorovich one, QOT on $\R$ is stable, and the independent coupling is rarely an optimizer.

To discuss the finiteness of the transport cost in QOT, denote by $   \mathcal C(\mu,\nu)$ the set
\begin{align*}
    \mathcal C(\mu,\nu) =\left\{c: (\mathfrak X \times \mathfrak Y )^2 \to \R ~\Big|~
    \genfrac{}{}{0pt}{}{c(x,y,x',y') \ge f(x,x') + g(y,y') \mbox{ everywhere}}{\mbox{for some $f\in L^1(\mu\otimes \mu)$ and $g\in L^1(\nu\otimes \nu)$}}\right\}.
\end{align*}
Note that a cost function $c$ is in $   \mathcal C(\mu,\nu) $ 
if it is bounded from below, or if  it is lower semi-continuous and 
$\mu$ and $\nu$ are compactly supported. 
The next remark is immediate. 
\begin{fact}
\label{remark:finite}
If $c\in \mathcal C(\mu,\nu)$, then
the infimum of the quadratic-form  transport cost in \eqref{eq:QOT cost} is well-defined and not $-\infty$.
\end{fact}
The next result gives conditions under which minimizers of QOT exist and under which Monge is equivalent to Kantorovich. Similar results have been established for special cases such as the GW distance (\citet[Corollary 10.1]{memoli2011gromov} and \citet[Theorem 3.2]{memoli2024comparison}; see Example \ref{ex:GW} for the formulation) and its extensions (e.g., \citet[Theorem 2]{bauer2024z}). We also refer to Section 3 of \cite{memoli2024comparison} for further results that compare the Monge and Kantorovich problems in the case of GW costs.

\begin{proposition}\label{prop:existence uniqueness}
    Suppose that $c\in \mathcal C(\mu,\nu)$ is lower semi-continuous. Then a minimizer of \eqref{eq:QOT cost} exists. In particular, if $\mu$ is atomless, $c$ is continuous, and $\X,\Y$ are compact, then
    $$\min_{\pi\in\Pi(\mu,\nu)}\iint c\,\d\pi\otimes\d\pi={\inf_{T\in\T(\mu,\nu)}\iint c(x,T(x),x',T(x'))\,\d\mu(x)\,\d\mu(x').}$$
\end{proposition}

The minimizer in Proposition \ref{prop:existence uniqueness} may not be unique even in many non-trivial cases, which we will see later.

We next show that, similar to classic OT, QOT satisfies stability with respect to the marginals.

\begin{proposition}\label{prop:stability}
    Suppose that $\mu,\nu\in\M(\R)$ and $\mu_n\to\mu,~\nu_n\to\nu$ weakly. Let $c:\R^4\to\R$ be a continuous function satisfying the uniform integrability condition
    \begin{align}
        \sup_{\pi\in\Pi(\mu,\nu)}\iint |c(x,y,x',y')|^{1+\delta}\d\pi(x,y)\,\d\pi(x',y')<\infty.\label{eq:uniform integrable}
    \end{align}
    for some $\delta>0$. 
    Let $\pi_n\in\Pi(\mu_n,\nu_n)$ be any QOT minimizer with cost function $c$. Then the sequence $\{\pi_n\}_{n\in\N}$ admits weak limit points in $\Pi(\mu,\nu)$ and every weak limit point of $\{\pi_n\}_{n\in\N}$ is a QOT minimizer with cost function  $c$ and marginals $\mu,\,\nu$.
\end{proposition}

Stability is crucial in classic OT theory as it ensures that numerical algorithms for solving the OT problem converge consistently. On the other hand, numerically computing or approximating the QOT solution remains a difficult task, as the discretized version remains an NP-hard problem (\cite{loiola2007survey}), a fact already noted in \citet[Remark 4.6]{memoli2011gromov}. The recent work of \cite{kravtsova2024np} also indicates the NP-hardness of the GW distance. As a discrete Monge version of QOT, QAP provides feasible heuristic algorithms; see \cite{burkard2012assignment} and \cite{cela2013quadratic}. However, one needs to be careful here since discrete QOT may not always have a Monge minimizer even if a bijective transport map exists (see Example \ref{ex:intro}). We leave the computational aspects of QOT for further investigation.

Lower bounds on  transport costs in QOT can be easily obtained from classic OT by either a two-step optimization or an optimization over aggregated marginals, which we summarize in the next two simple results.
Let $\C_c(\mu,\nu)$ denote the classic optimal transport cost from $\mu$ to $\nu$ with cost function $c$, i.e.,
$$\C_c(\mu,\nu):=\inf_{\pi\in\Pi(\mu,\nu)}\int c(x,y)\,\d \pi(x,y). $$

\begin{proposition}\label{prop:lb from conditioning}
Suppose that $c\in \mathcal C(\mu,\nu)$. It holds that
$$\inf_{\pi\in\Pi(\mu,\nu)}\iint c\,\d\pi\otimes\d\pi\geq \C_{\hat{c}}(\mu,\nu),$$
where $\hat{c}(x,y)=\C_{c_{x,y}}(\mu,\nu)$ and $c_{x,y}(x',y')=c(x,y,x',y')$.    
Moreover, if there exists $\pi_*\in \Pi(\mu,\nu)$ such that $\pi_*$ is the optimal coupling for both cost functions $\hat{c}$ and $c_{x,y}$ for $\pi_*$-a.e.~$(x,y)$, then $\pi_*$ is a minimizer of the QOT problem with cost function $c$.
\end{proposition}

\begin{proposition}\label{prop:lb from multimarginal}
Suppose that the cost function $c$ is type-XX, i.e., of the form
$$c(x,y,x',y')=h(f(x,x'),g(y,y'))$$
for some $h:\R^2\to\R,~f:\X^2\to\R$, and $g:\Y^2\to\R$. 
    It holds  that
    \begin{align*}
        &\inf_{\pi\in\Pi(\mu,\nu)}\iint c\,\d\pi\otimes\d\pi\geq \inf_{\hat{\pi}\in\Pi(\mu_f,\nu_g)}\int h(\xi,\zeta)\,\d\hat{\pi}(\xi,\zeta)=\C_{h}(\mu_f,\nu_g),
    \end{align*}
    where $\mu_f$ is the law of $f(X,X')$ for $X,X'$ independent following law $\mu$, and $\nu_g$ is the law of $g(Y,Y')$ for $Y,Y'$ independent following law $\nu$.
\end{proposition}

 The  lower bounds in Propositions \ref{prop:lb from conditioning} and \ref{prop:lb from multimarginal} are generally not sharp, but they are useful in proving the optimality of some transport plans. The essential ideas of these lower bounds are present in the literature on QAP and GW distances. More precisely, Proposition \ref{prop:lb from conditioning} generalizes Lawler's lower bound on QAP (\cite{burkard2012assignment}) and is known as the ``third lower bound'' in the GW setting (\cite{memoli2007use,memoli2011gromov}); Proposition \ref{prop:lb from multimarginal} is an extension of the ``second lower bound'' for the GW distance. Examples \ref{ex:GW} and \ref{ex:QAP} detail the connections between GW distance and QAP to QOT.

\section{Explicit solutions between measures on the real line}\label{sec:explicit}

In this section, we discuss a few instances where the QOT problem for $\mu,\nu \in \M(\R)$ allows for an explicit solution. 
Most of our results will be built on the lower bounds obtained in Section \ref{sec:general}. More precisely, the strategy is to show that certain lower bounds are achieved by specific transport plans (such as the comonotone coupling). Further results where the minimizer is attained by the diamond transport will be discussed in Section \ref{sec:diamond}.

\subsection{Type-XY product of two quadratic cost functions}

A natural class of cost functions to consider in OT theory is the quadratic ones. For instance, martingale optimal transport with a quadratic cost function is trivial. In the QOT framework, we consider a cost function $c(x,y,x',y')$ that is a quadratic function of four variables. After adding   terms that do not depend on the coupling, any such cost function $c$ is equivalent to one of the form $c(x,y,x',y')=f(x,y)g(x',y')$, where $f,g$ are quadratic functions of the two variables. 

We describe an algorithm that explicitly solves QOT problems whose cost function is of the form $c(x,y,x',y')=f(x,y)g(x',y')$, where $f,g$ are quadratic. This cost function is type-XY. \citet[Theorem 4.2.4]{vayer2020contribution} studies a special case $c(x,y,x',y')=xyx'y'$. Recall the notation  $\pi_{\rm x}^\lambda=\lambda\pi_{\rm com}+(1-\lambda)\pi_{\rm ant}$ for $\lambda\in[0,1]$, which is  called an X-transport if $\lambda\in (0,1)$.

\begin{proposition}\label{ex:quadratic cost}
   \sloppy  Suppose that $\mu,\nu\in\M(\R)$. If the cost function $c$ is given by \begin{equation}
   \label{eq:QC}
       c(x,y,x',y')=f(x,y)g(x',y'), \mbox{  ~~ where $f,g$ are quadratic functions},
   \end{equation}
   then there exists a QOT minimizer  $\pi_{\rm x}^\lambda$ for some $\lambda\in [0,1]$.
   Moreover, if $\mu,\nu$ are not degenerate, then   
      every    $\pi_{\rm x}^\lambda$ 
   minimizes the quadratic-form transport cost  for some  cost function in \eqref{eq:QC}  uniquely among the class $(\pi_{\rm x}^\lambda)_{\lambda \in [0,1]}$.  
\end{proposition}

The proof of Proposition \ref{ex:quadratic cost} contains an algorithm that explicitly solves such a QOT problem. We illustrate it with the following example.

\begin{example}    
    Consider $\mu,\nu$ both distributed as $\mathrm{N}(0,1)$ with the cost function given by $c(x,y,x',y')=-(x+y)^2(2x'-y')^2$. A standard computation yields that for $\pi\in\Pi(\mu,\nu)$,
    \begin{align*}
        \E_{\pi\otimes\pi}[c(X,Y,X',Y')]&=-\E[(X+Y)^2]\E[(2X'-Y')^2]\\
        &=-(2+2\E[XY])(5-4\E[X'Y'])=8\E[XY]^2-2\E[XY]-10.
    \end{align*}
    Since the quadratic function $z\mapsto 8z^2-2z-10$ is minimized at $z=1/8$, we see that if $\cov(X,Y)=1/8$, the law $\pi$ of $(X,Y)$ is a QOT minimizer. This is achieved, for example, by $\pi=(9/16)\pi_{\mathrm{com}}+(7/16)\pi_{\mathrm{ant}}$, where $(X,X)\lawis\pi_{\mathrm{com}}$ and $(X,-X)\lawis\pi_{\mathrm{ant}}$. On the other hand, since the range of $\E[XY]$ is $[-1,1]$, the unique QOT maximizer is given by the antimonotone coupling $X=-Y$, where $\E[XY]=-1$.
\end{example}

{Proposition \ref{ex:quadratic cost} follows from the following more general result, which replaces the quadratic functions $f,g$ in Proposition \ref{ex:quadratic cost} by submodular functions $f,g$ that are identical up to a constant term.
\begin{proposition}\label{ex:quadratic cost-f}
   \sloppy  Suppose that $\mu,\nu\in\M(\R)$ and $f$ is submodular. If the cost function $c$ is given by \begin{equation}
   \label{eq:QC-f} c(x,y,x',y')=(\alpha_1 f(x,y)+C_1) (\alpha_2 f(x',y')+C_2), \mbox{  ~~ where $C_1,C_2,\alpha_1,\alpha_2\in \R$},
   \end{equation}
   then there exists a QOT minimizer  $\pi_{\rm x}^\lambda$ for some $\lambda\in [0,1]$.    Moreover, if $\mu,\nu$ are not degenerate, then   
      every    $\pi_{\rm x}^\lambda$ 
   minimizes the quadratic-form transport cost  for some  cost function in \eqref{eq:QC-f}  uniquely among the class $(\pi_{\rm x}^\lambda)_{\lambda \in [0,1]}$.  
\end{proposition}
}

\subsection{Jointly submodular cost functions}

Similarly to the classic OT problems, 
submodular and supermodular cost functions lead to explicit optimizers of QOT, which are the comonotone and antimonotone couplings ($\pi_{\mathrm{com}}$ and $\pi_{\mathrm{ant}}$),
as we present in the next result. 

\begin{theorem}\label{prop:comonotone}
   \sloppy  Suppose that $\mu,\nu\in\M(\R)$ and the cost function $c\in \mathcal C(\mu,\nu)$ satisfies that both $(x',y')\mapsto c(x,y,x',y')$ and $(x,y)\mapsto c(x,y,x',y')$ are submodular (resp.~supermodular) for every $x,y,x',y'\in\R$. 
     Then the comonotone (resp.~antimonotone) coupling is a minimizer.
\end{theorem}

\begin{example}
\label{ex:th-sub}
    We give several examples in which the conditions in Theorem \ref{prop:comonotone} are satisfied, including both type-XX and type-XY ones. 
    \begin{enumerate}[(i)]
    \item If $c(x,y,x',y')=c_1(x,y)c_2(x',y')$ where both $c_1,c_2$ are nonnegative and submodular, the submodularity condition in Theorem \ref{prop:comonotone} is clearly satisfied. In fact, a direct proof that the comonotone coupling is an optimizer follows from $\iint c\, \d \pi \otimes \d \pi = \int  c_1 \, \d \pi  \int  c_2 \,  \d \pi  $. 
\item Generalizing (i), suppose that $c(x,y,x',y')=h(c_1(x,y),c_2(x',y'))$ and $c\in\C(\mu,\nu)$, where both $c_1,c_2$ are submodular and componentwise increasing,  
and $h:\R^2\to \R$ is componentwise increasing and concave. We can check that the submodularity condition in Theorem \ref{prop:comonotone} is satisfied. Hence, the comonotone coupling is a minimizer. Similarly, the antimonotone coupling is a maximizer if $-c\in \C(\mu,\nu)$.  This includes, for instance, $c(x,y,x',y')= \min\{c_1(x,y),c_2(x',y')\}$, and $c(x,y,x',y')= (c_1(x,y)+c_2(x',y'))^p$ where $p\in(0,1)$ and $c_1,c_2$ nonnegative.     

   \item Suppose that $c(x,y,x',y')=h (|x-y|,|x'-y'|)$ and $c\in\C(\mu,\nu)$, where $h$ is componentwise increasing and convex.
It is elementary to check that   the function
   $(x,y)\mapsto h(|x-y|,a)$ for $a\in \R$ is submodular. By Theorem \ref{prop:comonotone}, the comonotone coupling is a minimizer.
   This includes,  for instance, $c(x,y,x',y')= \max\{|x-y|,|x'-y'|\}.$
        Note that for   $(X,Y)\lawis \pi$,
$$\iint  \max\{|x-y|,|x'-y'|\}\, \d \pi(x,y)\, \d \pi(x',y') =
\int_0^\infty \left(1-\left( \p(|X-Y|\le  x)\right)^2\right) \d x,
$$
which is the transport cost in a distorted OT  problem (\cite{liu2023distorted})\footnote{For a nonnegative cost function $\tilde c:\mathfrak X\times \mathfrak Y \to \R$ and a \emph{distortion function} $\eta:[0,1]\to [0,1]$  increasing with $\eta(0)=0$ and $\eta(1)=1$, the distorted OT problem has transport cost  formulated by $\int_0^\infty \eta (\p(\tilde c(X,Y)>x))\,\d x $, with the classic OT corresponding to $\eta(t)=t$ on $[0,1]$; see \cite{liu2023distorted}.} with   distortion function $\eta: t\mapsto 1-(1-t)^2$ 
and    cost function $\tilde c:(x,y)\mapsto |x-y|$.

        \item Consider $c(x,y,x',y')=c_1(x,x')c_2(y,y'),$ where $c_1\geq 0$ is increasing in both arguments and $c_2\geq 0$ is decreasing in both arguments. Observe that a function $f(x,y)=a(x)b(y)$ is submodular if $a$ is increasing positive and $b$ is decreasing positive. Therefore, the function $(x',y')\mapsto c(x,y,x',y')$ is submodular by our assumption, and the same holds for $(x,y)\mapsto c(x,y,x',y')$.
 
        \item The function $|x+y-x'-y'|$ is related to the Gini coefficient (see Example \ref{ex:gini}) and satisfies the supermodularity condition in Theorem \ref{prop:comonotone}, since $(x,y)\mapsto |x+y+c|$ is supermodular for every $c\in\R$. More generally, the cost function $|x+y-x'-y'|^p$ (resp.~$|x-y+x'-y'|^p$) for $p\geq 1$ induces the antimonotone (resp.~comonotone) coupling as an optimizer. Similarly, the antimonotone coupling is an optimizer of QOT with cost function $|x+y+x'+y'|^p$ for $p\geq 1$. 
        
        \item The cost function $c(x,y,x',y')=\sgn(x-x')\,\sgn(y-y')$ defines   Kendall's  tau; see Example \ref{ex:kendall tau}. It is elementary to verify that $(x',y')\mapsto c(x,y,x',y')$ and $(x,y)\mapsto c(x,y,x',y')$ are supermodular, and hence the transport cost is maximized by the comonotone coupling and minimized by the antimonotone coupling. 
        \item Let $c(x,y,x',y')=\min\{x-x',y-y'\}$. Using that $(x,y)\mapsto \min\{x,y\}$ is supermodular, we see that the cost function $c$ satisfies the supermodularity condition in Theorem \ref{prop:comonotone}, and hence a minimizer is given by the comonotone coupling.

    \end{enumerate}
   
\end{example}

\subsection{Gromov--Wasserstein-type cost functions}

We now consider a family of type-XX cost functions, called the GW-type cost functions (see Example \ref{ex:GW} for  the GW distance).  
In what follows, we say that $\nu$ is an increasing (resp.~decreasing) location-scale transform of $\mu$ 
if $\nu=\mu\circ \ell^{-1}$ 
for some strictly increasing (resp.~decreasing) linear map $\ell:\R\to \R$; that is, $\ell(x)=ax+b$ for some $a>0$  (resp.~$a<0$) and $b\in \R$.

\begin{theorem}\label{prop:cocounter}
  Suppose that $\mu \in\M(\R)$, $\nu$ is an increasing  (resp.~decreasing) location-scale transform of $\mu$, and $h:\R_+^2\to\R$ is a submodular function. 
  Then the comonotone (resp.~antimonotone) coupling   is a minimizer of the QOT with the cost function $c$ given by $$c(x,y,x',y')=h(|x-x'|,|y-y'|).$$ 
  Moreover, such a minimizer is unique if it yields a finite transport cost, $h$ is strictly submodular, and $\mu$ is asymmetric. In the same setting except that $\mu$ is symmetric, the comonotone and antimonotone couplings
   are the only minimizers.
\end{theorem}
In particular,  an optimal coupling in Theorem \ref{prop:cocounter} is precisely given by the Monge map $x\mapsto \ell(x)$, that is, the linear transform connecting $\mu$ and $\nu$.

\begin{example}
\label{ex:47}
We give a few examples of QOT problems that satisfy the conditions in Theorem \ref{prop:cocounter}. Assume $\mu=\nu\in\M(\R)$ in all items.
\begin{enumerate}[(i)]
    \item Let $h(u,v)=\max\{u,v\}$, which is submodular. Theorem \ref{prop:cocounter} then implies that $\pi_{\mathrm{com}}$ is a minimizer for the cost function $c(x,y,x',y')=\max\{|x-x'|,|y-y'|\}$.  
    \item Let $h(u,v)=-u^{-\alpha}v^{-\alpha}$ for $\alpha\in(0,1/2)$. If $\mu$ has a uniformly bounded density, $$0\leq -\E_{\pi_{\mathrm{com}}\otimes\pi_{\mathrm{com}}}[h(|X-X'|,|Y-Y'|)]=\E[|X-X'|^{-2\alpha}]<\infty.$$
    In this case, Theorem \ref{prop:cocounter} implies that $\pi_{\mathrm{com}}$ (and $\pi_{\mathrm{ant}}$ if $\mu$ is symmetric) is the unique minimizer for the cost function $c(x,y,x',y')=-|x-x'|^{-\alpha}|y-y'|^{-\alpha}$.

    \item Let $h(u,v)=-u^\beta v^\beta$ where $\beta>0$. It is easy to verify that $h$ is submodular. If $\mu\in\P_{2\beta}(\R)$, 
    $$0\leq -\E_{\pi_{\mathrm{com}}\otimes\pi_{\mathrm{com}}}[h(|X-X'|,|Y-Y'|)]=\E[|X-X'|^{2\beta}]<\infty.$$
    It follows from Theorem \ref{prop:cocounter} that $\pi_{\mathrm{com}}$ (and $\pi_{\mathrm{ant}}$ if $\mu$ is symmetric) is the unique minimizer for the cost function $c(x,y,x',y')=-|x-x'|^{\beta}|y-y'|^{\beta}$. The same cost function is also investigated in \cite{beinert2023assignment} in the case $\mu\neq \nu$, where it is shown that the Monge minimizer may be far away from the comonotone coupling.

    \item The choice $h(u,v)=|u^q-v^q|^p$ corresponds to the $(p,q)$-GW transport cost, see \eqref{eq:GWpq}. One can verify that $h$ is submodular if $p\geq 1,q>0$ on $[0,\infty)^2$. Theorem \ref{prop:cocounter} then implies that in the case $\mu=\nu$, $\pi_{\mathrm{com}}$ (and $\pi_{\mathrm{ant}}$ if $\mu$ is symmetric) is a minimizer for \eqref{eq:GWpq}. 
    This aligns with the intuition that the GW distance measures distances between metric measure spaces.    
\end{enumerate}
    
\end{example}


\begin{remark}
    Theorem \ref{prop:cocounter} extends naturally to more general Polish spaces, where $\X=\Y$, $\nu$ is a lateral shift of $\mu$, and $c(x,y,x',y')=h(d(x,x'),d(y,y'))$. In this case, the lateral shift is always a Monge minimizer, but the uniqueness of the minimizer may depend on the geometry of the Polish space and the measures $\mu,\nu$.
\end{remark}

\subsection{A special class of separable cost functions and the V-transport}

In this section, we prove that the V-transport mentioned in the Introduction serves as a minimizer for a special class of separable type-XX cost functions. We first rigorously define the V-transport. Recall that $Q_\mu $ is the left quantile function of $\mu\in\M(\R)$.

\begin{definition}\label{def:wedge}
A coupling $(X,Y)$ with marginals $\mu$ and $\nu$, or its joint distribution, is the \textit{V-transport} if 
$(X,Y)\laweq (Q_\mu(U),Q_{\nu}(|2U-1|))$,
where $U\lawis\mathrm{U}$. In this case, we denote the law of $(X,Y)$ by $\pi_{\rm v}$.
\end{definition}

For instance, if $\mu,\nu\lawis\mathrm{U}$, then $\pi_{\rm v}$ is the distribution of $(U,|2U-1|)$ where $U\lawis\mathrm{U}$; see Figure \ref{fig:transport}(b). 
If $\nu$ is atomless with median $m_{\nu}$, then $\pi_{\rm v}$ is the arithmetic average of the antimonotone coupling of $\mu$ and $2\nu|_{(-\infty,m_\nu]}$, and the comonotone coupling of $\mu$ and $2\nu|_{(m_\nu,\infty)}$.

\begin{theorem}\label{prop:wedge}
     Suppose   $\mu=\mathrm{U}(a,b)$ for some $a<b$, $ \nu\in\M(\R)$,  and the cost function $c$  has the form
    \begin{align}
        c(x,y,x',y')=f(|x-x'|)g(y,y'),\label{eq:special}
    \end{align}
    \sloppy where $g$ is right-continuous, 
    increasing in both arguments, supermodular, and satisfies
    $$\lim_{y\to -\infty}g(y,y')=\lim_{y'\to -\infty}g(y,y')=0,$$
     and $h$ is nonnegative, right-continuous, and increasing. Then the V-transport is a minimizer.
\end{theorem}

Theorem \ref{prop:wedge} requires that $\mu$ is uniformly distributed on a compact interval. For a general atomless $\mu$, we can transform the marginal by using $f(|F_{\mu}(x)-F_{\mu}(x')|)$
instead of $f(|x-x'|)$ in \eqref{eq:special}, and the same result applies.

The conditions on $g$ hold, for instance, if $g(y,y')=\phi(y)\psi(y')$ where both $\phi,\psi$ are increasing, right-continuous, and satisfy $\lim_{y\to-\infty}\phi(y)=\lim_{y\to-\infty}\psi(y)=0$.

The QAP version of Theorem \ref{prop:wedge} is contained in \cite{burkard1998quadratic}. However, the Monge assumptions of QAP cannot be relaxed in general, and hence we cannot directly apply stability (Proposition \ref{prop:stability}) to solve the corresponding QOT.

\section{The diamond transport}
\label{sec:diamond}

In this section, we systematically study a new transport, the diamond transport, which turns out to be a minimizer for several classes of type-XX QOT cost functions. 
Its definition is presented below. 

\begin{definition} \label{def:diamond}
Let
$
D=\{(x,y)\in [0,1]^2: 
|y-1/2| + |x-1/2|=1/2 \}.
$
 The \emph{diamond copula} $C_{\rm dia}$ is the cdf of the uniform distribution on $D$.  
 The \emph{diamond transport} $\pi_{\mathrm{dia}}\in \Pi(\mu,\nu)$ between $\mu,\nu$ is the law of  $(Q_{\mu} (U),Q_{\nu}(V))$ where $(U,V)\lawis C_{\rm dia}$. 
\end{definition}

 In terms of cdf,
 $\pi_{\rm dia}\in \Pi(\mu,\nu)$ can be expressed as \begin{align}
     \label{eq:dia-sklar}
 F_{\pi_{\rm dia}}(x,y)= C_{\rm  dia}(F_\mu (x),F_{\nu}(y)),~~~x,y\in\R. 
\end{align}
Denote by $a\wedge b$ the minimum of $a,b$
and by $a\vee b$ the maximum of $a,b$. Moreover, let $$a\diamond b=\frac{a}{2}+\frac{b}{2}-\frac{1}{4}.$$
By direct calculation, 
the diamond copula has an explicit cdf formula
\begin{align}
C_{\rm  dia}( u,v) = \begin{cases}
 (  u\diamond v)_+ & (u,v)\in [0,1/2]^2
    \\    (  u\diamond v)\wedge v  & (u,v)\in (1/2,1]\times [0,1/2] 
       \\   (  u\diamond v) \wedge u  & (u,v)\in [0,1/2] \times  (1/2,1]    \\   (  u\diamond v) \vee (u+v-1)  & (u,v)\in   (1/2,1]^2.
\end{cases}
\label{eq:dia-cop}
\end{align}
In particular, $C_{\rm dia}(u,v)=u\diamond v$ when $(u,v)$ is in the area inside $D$.
In the case $\mu=\nu=\mathrm{U}$, the diamond copula coincides with the diamond transport; see Figure \ref{fig:cdf} for an illustration.

\begin{figure}[t]
    \centering
  \begin{tikzpicture}
    \begin{axis}[
        axis lines=middle,
        width=8cm,
        height=8cm,
        enlargelimits,
        ymin=0, ymax=1,
        xmin=0, xmax=1,
        ytick=\empty,
        xtick=\empty,
        axis equal 
    ]
        \addplot[name path=F,blue,domain={0:0.5}] {0.5-x} node[pos=0.8, above] {$D$};
        \addplot[name path=F,black,dashed,domain={0:1}] {1} node[pos=0.8, above] {};
        \addplot[name path=F,blue,domain={0:0.5}] {0.5+x};
        \addplot[name path=F,blue,domain={0.5:1}] {x-0.5};
        \addplot[name path=F,blue,domain={0.5:1}] {1.5-x};

        \node at (axis cs:0.82,0.88) {$=u+v-1$};
        \node at (axis cs:0.5,0.5) {$C_{\rm dia}(u,v)=u\diamond v$};
        \node at (axis cs:0.85,0.12) {$=v$};
        \node at (axis cs:0.15,0.88) { $=u$};
        \node at (axis cs:0.15,0.12) { $=0$};

        \node at (axis cs:-0.03,-0.04) {0};
        \node at (axis cs:-0.03,1) {1};
        \node at (axis cs:1,-0.04) {1};
    \end{axis}

    \node at (0.45,6.65) {$v$};
    \node at (6.65,0.45) {$u$};

   
    \draw [dashed] (5.9,0.55) -- (5.9,5.9);
\end{tikzpicture}

    \caption{The value $C_{\rm dia}(u,v)$ of the diamond copula, illustrated by distinct values in different regions. The blue shape is the support $D$ of the diamond copula, which also indicates transitions of the cdf across different regions.}
    \label{fig:cdf}
\end{figure}
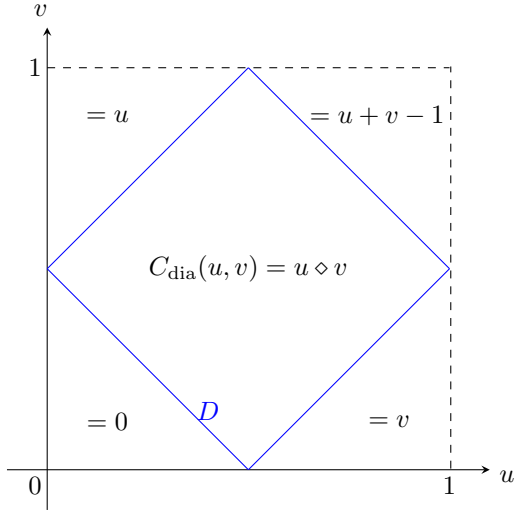

\subsection{The rectangular cost function}

We now consider the \emph{rectangular cost function}  $c(x,y,x',y')=|(x-x')(y-y')|$, which is the area of the rectangle formed by the two vertices $(x,y)$ and $(x',y')$.
By Fact \ref{rem:marginals},
 this cost function is equivalent to the ones in two other problems: 
\begin{enumerate}[(a)]
    \item minimizing the transport cost is equivalent to the problem of inequality minimization  in the Introduction with cost function $ (\theta_1|x-x'|+\theta_2|y-y'|)^2$  
in \eqref{eq:motivation1} with $\theta_1,\theta_2>0$;
    \item maximizing the transport cost is equivalent to computing the $(2,1)$-GW distance in Section \ref{sec:exs} defined  in  \eqref{eq:GWpq}.
\end{enumerate} 
In the next theorem, we see that the diamond transport uniquely solves the QOT problem. A family of more general cost functions will be studied in Section \ref{sec:p-cost}, where we prove analogous results.


\begin{theorem}\label{prop:pidia}
 Let $\mu,\nu\in\P_1(\R)$. For the rectangular cost function $c$   given by 
 $$c(x,y,x',y')=|(x-x')(y-y')|,$$
the unique minimizer of the QOT problem is the diamond transport $\pi_{\mathrm{dia}}$.
\end{theorem}

\begin{proof}
For $(x,y),(x',y')\in\R^2$, denote by $[(x,y),(x',y')]$ the unique closed rectangle in $\R^2$ whose sides are parallel to the $xy$-axes and two of whose corners are given by $(x,y),(x',y')$ if $x\neq x'$ and $y\neq y'$, and the empty set otherwise. It holds that for $\pi\in\Pi(\mu,\nu)$,
\begin{align}
   & \iint |(x-x')(y-y')|\,\d\pi(x,y)\,\d\pi(x',y')\nonumber\\
   &=\iint\iint \bone_{\{(u,v)\in [(x,y),(x',y')]\}}\d\pi(x,y)\,\d\pi(x',y')\, \d u\,\d v\nonumber\\
    &=\iint\pi\otimes\pi(\{(x,y,x',y')\in\R^4: (u,v)\in [(x,y),(x',y')]\})\,\d u\,\d v.\label{eq:integ}
\end{align}
Our goal is to show that the integrand in \eqref{eq:integ} is uniquely minimized for all $(u,v)$ by ${\pi}=\pi_{\mathrm{dia}}$, which suffices for our purpose: since $\mu,\nu\in\P_1(\R)$, the transport cost for the independent coupling is finite, and hence so is the transport cost for $\pi_{\mathrm{dia}}$. 
For a fixed $(u,v)\in\R^2$ such that $\mu$ has no atom at $u$ and $\nu$ has no atom at $v$, define $A(u,v)=\pi((-\infty,u]\times(-\infty,v])$. Also, denote by $F_\mu,F_\nu$ the distribution functions of $\mu,\nu$, so $F_\mu$ is continuous at $u$ and $F_\nu$ is continuous at $v$. Using $\pi\in\Pi(\mu,\nu)$, we have 
\begin{align*}
    &\pi\otimes\pi(\{((x,y),(x',y'))\in\R^4: (u,v)\in [(x,y),(x',y')]\})\\
    &= 2\big(A(u,v)(1-F_\mu(u)-F_\nu(v)+A(u,v))+(F_\mu(u)-A(u,v))(F_\nu(v)-A(u,v))\big)\\
    &=4A(u,v)^2-(4F_\mu(u)+4F_\nu(v)-2)A(u,v)+2F_\mu(u)F_\nu(v).
\end{align*}
This is a quadratic function in $A(u,v)$, which is uniquely minimized on $\R$ at $A(u,v)=(2F_\mu(u)+2F_\nu(v)-1)/4 =F_\mu(u)\,\diamond F_{\nu}(v)$. Note that the feasible region for $A(u,v)$ is $(F_\mu(u)+F_{\nu}(v)-1)_+\leq A(u,v)\leq \min\{F_\mu(u),F_\nu(v)\}$. 
Hence, the unique minimizer is given by
$$
A(u,v) = \min\big\{\max\{ F_\mu(u)\,\diamond F_{\nu}(v), F_\mu(u)+F_{\nu}(v)-1,0\},F_\mu(u),F_\nu(v)\big \}.
$$
By \eqref{eq:dia-sklar} and \eqref{eq:dia-cop}, this is the cdf of the diamond transport.
\end{proof}

As explained in item (a) at the beginning of this section, 
Theorem \ref{prop:pidia} fully solves problem \eqref{eq:motivation1} described in the Introduction in the Kantorovich setting; that is, the diamond transport yields a minimum inequality quantified by \eqref{eq:motivation1}. 
For applications in the Monge setting,
this result also leads to approximately optimal transport maps by approximating the diamond transport with a permutation map, since the cost function is continuous.  
A notable feature of this solution, different from the comonotone or antimonotone coupling, 
is $\E[V\mid U]=1/2$ when $U,V\lawis C_{\rm dia}$. 
Hence, in the financial policy application described in the Introduction, if the marginal distributions are uniform, individuals across different wealth levels get the same benefit on average under the optimal policy.

In Table \ref{table:numbers2}, we report some numerical values for the cost function $c(x,y,x',y')=(|x-x'|+|y-y'|)^2$, with different couplings and marginal distributions. 
We can observe that, while $\pi_{\rm dia}$ uniquely minimizes the transport cost as shown in Theorem \ref{thm:diamond}, the situation is unclear 
for the maximizers. 

\begin{table}[t]\centering\setlength\tabcolsep{3.5pt}\renewcommand\arraystretch{2}
\small
  \noindent\makebox[\textwidth]{%
    \begin{tabular}{l|*{4}{c|}c} 
      \diagbox[width=\dimexpr \textwidth/6+2\tabcolsep\relax, height=1.5cm]{ Marginals }{Coupling}
                   & $\pi_{\rm com}$ & $\pi_{\rm ant}$ & $\pi_{\rm ind}$ & $\pi_{\rm x}^{0.5}$ & $\pi_{\rm dia}$\\
      \hline
      $\mathrm{U}(0,1)$, $\mathrm{N}(0,1)$ &(3.296) &   (3.296) &2.916& 3.051 & \textbf{2.884}\\
      \hline
       $\mathrm{U}(0,1)$, $\mathrm{U}(0,1)$ & (0.667) & (0.667)&0.555& 0.583    &\textbf{0.547}\\
      \hline   $\mathrm{N}(0,1)$, $\mathrm{N}(0,1)$ 
      & (8.001) &(8.001)  & 6.543 & 7.273 & \textbf{6.439}\\
      \hline
             $\mathrm{Exp}(1)$, $\mathrm{Exp}(1)$ & (7.998) & 6.580 &5.998& 6.772    &\textbf{5.763} \\ 
      \hline     
      $\mathrm{U}(0,1)$, $\mathrm{Exp}(1)$ & 3.166 &  (3.168) & 2.832 & 2.963     & \textbf{2.775} \\ 
      \hline   
          $\mathrm{N}(0,1)$, $\mathrm{Exp}(1)$ & 7.612 &  (7.615) & 6.255 & 7.002    & \textbf{6.007} \\ 
      \hline   
    \end{tabular}
  }\\~%
  \caption{
  Quadratic-form transport costs for different couplings $\pi_{\rm com}$, $\pi_{\rm ant}$, $\pi_{\rm ind}$, $\pi_{\rm x}^{0.5}$, and $\pi_{\rm dia}$, with the cost function $c(x,y,x',y')=(|x-x'|+|y-y'|)^2$. The marginals are chosen from $\mathrm{N}(0,1)$, $\mathrm{U}(0,1)$, or $\mathrm{Exp}(1)$.
  The smallest transport cost in each row 
is marked in bold and the largest in brackets.
  Transport costs are computed using Monte Carlo simulation with sample size $10^7$.}
  \label{table:numbers2}
\end{table}

The diamond transport is in general not Monge. Obtaining a Monge QOT minimizer for generic discrete marginals seems a technically challenging task, which we do not pursue here.

 One may recall from Example \ref{ex:intro} that for the marginals $\mu,\nu$ being Bernoulli($1/2$), the independent coupling uniquely minimizes the quadratic-form transport cost with the rectangular cost function.
It may be useful to note that for this particular pair of discrete marginals, 
the independent coupling coincides with the diamond coupling because the class $\Pi(\mu,\nu)$ only has one parameter.

\begin{remark}\label{rem:maximizer}
    If $\mu$ is an increasing (resp.~decreasing) location-scale transform of $\nu$, then a maximizer of the transport cost with the cost function $c$ in Theorem \ref{prop:pidia}  
 is the comonotone (resp.~antimonotone) coupling. Indeed, the function $(x,y)\mapsto -xy$ is submodular and the claim follows from Theorem \ref{prop:cocounter}. 
\end{remark}

\begin{remark}
The absolute value in the cost function $c$ in Theorem \ref{prop:pidia} is important.
If $c$ is specified by $c(x,y,x',y')=(x-x')(y-y')$ (or equivalently, if the objective is $\cov(X,Y)$; c.f.~Example \ref{ex:cov}), then the minimizer is the antimonotone coupling instead of the diamond transport.
\end{remark}

\subsection{A class of type-XX cost functions with convex QOT}
\label{sec:convex}

We next provide a general result on the diamond transport as the unique minimizer. An example is the cost function $$c(x,y,x',y')=e^{-\alpha((x-x')^2+(y-y')^2)}$$ for $\alpha\in(0,1/2]$. Specifically, one of our assumptions in this result is that the quadratic-form transport cost is convex in the transport plan $\pi$. Schoenberg's theory of complete monotonicity provides a convenient sufficient condition for such convexity. A nonnegative continuous function $\phi:\R_+\to\R_+$ is \textit{completely monotone} if $\phi$ is $C^\infty$ on $(0,\infty)$ and satisfies $(-1)^n \phi^{(n)}(u)\geq 0$ for $n\geq 0,\,u>0$ (see \cite[Section 4.6]{berg1984harmonic}). In particular, $\phi$ is bounded and decreasing. 
We use the standard calculus notation $\phi'$  
for the first derivative $\phi^{(1)}$, which should not be confused with the apostrophe in $x',y'$.

\begin{theorem}\label{thm:diamond}
         Suppose that both $\mu,\nu\in\M(\R)$ are symmetric and the cost function $c$ is given by
    \begin{align}
        c(x,y,x',y')=\phi((x-x')^2)\phi((y-y')^2)\label{eq:phi^2}
    \end{align}
    for some completely monotone function $\phi:\R_+\to \R$ satisfying 
    \begin{align}
        \phi'(u)+2u\phi''(u)\leq 0, ~~~u\in \mathcal{D},\label{eq:phi condition}
    \end{align} 
    where $\mathcal{D}\subseteq\R_+$ is such that $(\mu\otimes \mu)\{|x-x'|^2\in \mathcal{D}\}=(\nu\otimes \nu)\{|y-y'|^2\in \mathcal{D}\}=1$.
    Then the diamond transport $\pi_{\mathrm{dia}}$ is a minimizer of the QOT problem, and the unique minimizer if $\phi$ is non-constant.
\end{theorem}

\begin{remark}
    The cost function \eqref{eq:phi^2} can be written as $c(x,y,x',y')=h(|x-x'|,|y-y'|)$, where $h(u,v)=\phi(u^2)\phi(v^2)$. Since $u,v\geq 0$ and $\phi'\leq 0$, we have $h_{uv}\geq 0$ and hence $-h$ is submodular. By 
    arguing in the same way as Remark \ref{rem:maximizer}, we see that under the setting of Theorem \ref{thm:diamond},     if $\mu$ is a location-scale transform of $\nu$, then the comonotone and antimonotone couplings are both maximizers.
\end{remark}
 
\begin{example}\label{ex:dia}
    We give three classes of examples of the type-XX cost function $c$ where the conditions in Theorem \ref{thm:diamond} are satisfied and the unique minimizer is given by $\pi_{\mathrm{dia}}$.
    \begin{enumerate}[(i)]
        \item Let $\phi(u)=e^{-\alpha u}$ where $\alpha\in(0,1/2]$. In this case, $$c(x,y,x',y')=e^{-\alpha((x-x')^2+(y-y')^2)}=e^{-\alpha\|(x,y)-(x',y')\|^2}.$$ The completely monotone condition is evident by definition. To check \eqref{eq:phi condition}, simply note that 
        $$ \phi'(u)+2u\phi''(u)=e^{-\alpha u}(-\alpha+2u\alpha^2)\leq 0$$
        if $u\leq 1/(2\alpha)$. Thus, the conditions in Theorem \ref{thm:diamond} are satisfied if $\mu,\nu$ are supported on $[-1/(2\alpha),1/(2\alpha)]$.
        \item Suppose that $\mu,\nu$ are supported on $[-1/2,1/2]$. Let $\phi(u)=(\beta+u)^{-\gamma}$, where $\gamma>0$ and $\beta>2\gamma+1$. This leads to the cost function $$c(x,y,x',y')=(\beta+|x-x'|^2)^{-\gamma}(\beta+|y-y'|^2)^{-\gamma}.$$ Again, the completely monotone condition is evident from the definition. To see \eqref{eq:phi condition}, we observe that
        $$\phi'(u)+2u\phi''(u)=\gamma (\beta+u)^{-\gamma-2}(2u(\gamma+1)-(\beta+u))\leq 0 $$
        for $u\in[0,1]$. 

        \item Suppose that $\mu,\nu$ are supported on $[-L,L]$ where $L>0$. Let $p\in(1/2,1)$ and $\phi(u)=e^{-\alpha u^p}$, where $0<\alpha\leq (2p-1)/(2p(4L^2)^p)$. The cost function is then given by 
        \begin{align}
            c(x,y,x',y')=e^{-\alpha(|x-x'|^{2p}+|y-y'|^{2p})}.\label{eq:p-cost}
        \end{align}
        The complete monotonicity of $\phi$ follows from Exercise 55.1 of \cite{sato1999levy} (which can be seen from the infinite divisibility of the Weibull distribution with parameter in $(0,1]$, a result in \cite{steutel1970preservatoin}). To verify \eqref{eq:phi condition}, we compute that for $u\in[0,4L^2]$,
        $$\phi'(u)+2u\phi''(u)=\alpha pu^{p-1}e^{-\alpha u^p}(1+2\alpha pu^p-2p)\leq 0.$$

    \end{enumerate}
   
\end{example}

To prove Theorem \ref{thm:diamond}, we introduce a technical lemma.
Define
 \begin{align}
        \tilde{c}(x,y):= \int c(x,y,x',y')\,\d\pi_{\mathrm{dia}}(x',y').\label{eq:tilde c}
    \end{align}



\begin{lemma}\label{lemma:1/2}
    Assume the same setting as Theorem \ref{thm:diamond}.  Then $\tilde{c}$ is supermodular on the first and third quadrants, and submodular on the second and fourth quadrants.
\end{lemma}

The proof of Lemma \ref{lemma:1/2} involves detailed analysis and is deferred to Appendix \ref{app:sec5}.

\begin{proof}[Proof of Theorem \ref{thm:diamond}]
We first claim that the resulting QOT problem is convex in $\pi$. Since $\phi$ is completely monotone, {the function $\psi(u):=\phi(u^2)$ is a continuous positive definite function on $\R$ by Schoenberg's theorem (Theorem 4.6.13 and Example 5.1.3 of \cite{berg1984harmonic}). In this way, we write $c(x,y,x',y')=\psi(x-x')\psi(y-y')$ where $\psi$ is positive definite, so $(x,x')\mapsto \psi(x-x')$ and $(y,y')\mapsto \psi(y-y')$ are positive definite kernels.} By the Schur product theorem, the product of two positive definite {kernels} is also positive definite (see Theorem 3.1.12 of \cite{berg1984harmonic}). Therefore, $c(x,y,x',y')$ is a positive definite kernel in the two variables $(x,y),(x',y')\in\R^2$. Proposition \ref{prop:convex?} then implies the QOT problem is convex in $\pi$.

Since the cost function \eqref{eq:phi^2} is translation-invariant, we may assume that $\mu,\nu$ are both symmetric along $0$. We next verify 
    \begin{align}
        \tilde{c}(x,y)=\tilde{c}(-x,y)=\tilde{c}(x,-y)=\tilde{c}(-x,-y).\label{eq:4c}
    \end{align}
      By the symmetry of $\pi_{\mathrm{dia}}\in\Pi(\mu,\nu)$, we have 
\begin{align*}
    \tilde{c}(x,y)&=\int c(x,y,x',y')\,\d\pi_{\mathrm{dia}}(x',y')\\
    &=\int c(x,y,-x',y')\,\d\pi_{\mathrm{dia}}(x',y')=\int c(-x,y,x',y')\,\d\pi_{\mathrm{dia}}(x',y')=\tilde{c}(-x,y),
\end{align*}
where we have used \eqref{eq:phi^2} in the third equality. Similarly, $\tilde{c}(x,y)=\tilde{c}(x,-y)$, and hence \eqref{eq:4c} holds.

Consider the classic OT problem with cost function $\tilde{c}$ given by \eqref{eq:tilde c}. Denote a minimizer by $\pi_{\tilde{c}}$. By linearity of the classic OT problem and \eqref{eq:4c}, another minimizer is given by the symmetrized version $\hat{\pi}_{\tilde{c}}$ of $\pi_{\tilde{c}}$: the law of the uniform mixture of $(X,Y),\,(-X,Y),\,(X,-Y),$ and $(-X,-Y)$ where $(X,Y)\lawis \pi_{\tilde{c}}$. 
Therefore, $\hat{\pi}_{\tilde{c}}$ is symmetric along the $x$ and $y$ axes.
By cyclical monotonicity and Lemma \ref{lemma:1/2}, any minimizer is antimonotone on the first and third quadrants and comonotone on the second and fourth quadrants. This implies that $\hat{\pi}_{\tilde{c}}=\pi_{\mathrm{dia}}$, and hence ${\pi_{\mathrm{dia}}}$ is a minimizer of the OT problem with cost function $\tilde{c}$.  

Suppose that $\pi\in\Pi(\mu,\nu)$ satisfies
$$\iint c\,\d\pi\otimes\d\pi<\iint c\,\d\pi_{\mathrm{dia}}\otimes\d\pi_{\mathrm{dia}}.$$
For $\delta\in[0,1]$, let $\pi_\delta=(1-\delta)\pi_{\mathrm{dia}}+\delta \pi$. It follows by convexity of the QOT that $\delta\mapsto \iint c\,\d\pi_\delta\otimes\pi_\delta$ is convex in $\delta$, so there exists $\varepsilon>0$ such that for all $\delta\in[0,1]$ small enough,
\begin{align*}
    &\iint c\,\d\pi_{\mathrm{dia}}\otimes\d\pi_{\mathrm{dia}}-\delta\varepsilon\\
    &>\iint c\,\d\pi_\delta\otimes\pi_\delta\\
    &=(1-\delta)^2\iint c\,\d\pi_{\mathrm{dia}}\otimes\d\pi_{\mathrm{dia}}+2\delta(1-\delta)\iint c\,\d\pi_{\mathrm{dia}}\otimes\d\pi+\delta^2 \iint c\,\d\pi\otimes\d\pi\\
    &= \iint c\,\d\pi_{\mathrm{dia}}\otimes\d\pi_{\mathrm{dia}}+2\delta\bigg(\iint c\,\d\pi_{\mathrm{dia}}\otimes\d\pi-\iint c\,\d\pi_{\mathrm{dia}}\otimes\d\pi_{\mathrm{dia}}\bigg)+O(\delta^2).
\end{align*}
Therefore, letting $\delta\to 0$ yields
$$\iint c\,\d\pi_{\mathrm{dia}}\otimes\d\pi<\iint c\,\d\pi_{\mathrm{dia}}\otimes\d\pi_{\mathrm{dia}},$$
    contradicting $\pi_{\mathrm{dia}}$ being a minimizer of the OT problem with cost function $\tilde{c}$.

    The final claim on uniqueness follows from the fact that if $\phi$ is non-constant, the cost function $c$ is {a strictly positive definite kernel} in the two variables $(x,y),(x',y')\in\R^2$ (Theorem 3' of \cite{schoenberg1938metric}), and hence the transport cost is strictly convex in $\pi$ with a unique minimizer. 
\end{proof}

\subsection{The \texorpdfstring{$q$}--rectangular cost function}\label{sec:p-cost}

Applying Theorem \ref{thm:diamond}, we show that the diamond transport solves another class of QOT problems with the \emph{$q$-rectangular cost function} $|(x-x')(y-y')|^q$ for $1<q\leq 2$. 
Note that in these cases, the transport cost is not convex in $\pi$ in general, since the map $((x,y),(x',y'))\mapsto |(x-x')(y-y')|^q$ is not a positive definite kernel.
  QOT with this cost function is equivalent to the one with cost function 
$ (|x-x'|^q+|y-y'|^q)^2$, or equivalently, the maximization of the $(2,q)$-GW transport cost defined in  \eqref{eq:GWpq}.

\begin{theorem}\label{thm:p-cost}
 Let $q\in(1,2]$ and $\mu,\nu\in\P_{2+\delta}(\R)$ for some $\delta>0$. Suppose that $\mu,\nu$ are symmetric and the cost function $c$ is given by 
 $$c(x,y,x',y')=|(x-x')(y-y')|^q.$$ Then the diamond transport $\pi_{\mathrm{dia}}$ is a minimizer of the QOT problem. 
\end{theorem}

\begin{proof} Fix $q\in(1,2]$ and assume without loss of generality that $\mu,\nu$ are both symmetric around $0$. Assume first that $\mu,\nu$ are supported in $[-L,L]$ for some $L>0$. 
For $\alpha>0$, consider the cost function
$$c_\alpha(x,y,x',y'):=\frac{e^{-\alpha(|x-x'|^q+|y-y'|^q)}-1+\alpha(|x-x'|^q+|y-y'|^q)}{\alpha^2}.$$
    The QOT problem with cost function $c_\alpha$ is equivalent to that with cost function $e^{-\alpha(|x-x'|^q+|y-y'|^q)}$ (c.f.~\eqref{eq:p-cost}), since the term $|x-x'|^q+|y-y'|^q$ is QOT-irrelevant. By Theorem \ref{thm:diamond} and Example \ref{ex:dia}-(iii), for $0<\alpha\leq (q-1)/(q(4L^2)^{q/2})$, the unique minimizer is given by the diamond transport $\pi_{\mathrm{dia}}$.

    On the other hand, by the Taylor expansion, we have 
    $$c_\alpha(x,y,x',y')\to \frac{1}{2}(|x-x'|^q+|y-y'|^q)^2$$
    uniformly on $[-L,L]^4$ as $\alpha\to 0$. Therefore, the QOT problem with cost function $(|x-x'|^q+|y-y'|^q)^2/2$ also has $\pi_{\mathrm{dia}}$ as a minimizer. Removing the QOT-irrelevant terms, we see that this is also the case for the cost function $|(x-x')(y-y')|^q$, as desired.

    Now suppose that $\mu,\nu\in\P_{2+\delta}(\R)$. We may then apply the stability for QOT (Proposition \ref{prop:stability}) to approximate $\mu,\nu$ with measures with bounded supports. It remains to show \eqref{eq:uniform integrable}, i.e., 
    $$\sup_{\pi\in\Pi(\mu,\nu)}\iint |(x-x')(y-y')|^{q+\delta/2}\d\pi(x,y)\,\d\pi(x',y')<\infty.$$
    Indeed, this follows from Cauchy--Schwarz: there exists some constant $C(p,\delta)>0$ such that uniformly for $\pi\in\Pi(\mu,\nu)$,
    \begin{align*}
        \iint |(x-x')(y-y')|^{q+\delta/2}\d\pi(x,y)\,\d\pi(x',y')&=\E_{\pi\otimes\pi}[|X-X'|^{q+\delta/2}|Y-Y'|^{q+\delta/2}]\\
        &\leq \E_{\pi\otimes\pi}[|X-X'|^{2q+\delta}]^{1/2}\E_{\pi\otimes\pi}[|Y-Y'|^{2q+\delta}]^{1/2}\\
        &\leq C(p,\delta)\E_\mu[|X|^{2q+\delta}]^{1/2}\E_\nu[|Y|^{2q+\delta}]^{1/2}.
    \end{align*}
    This completes the proof.
\end{proof}

Different from the case of $q=1$ analyzed in Theorem \ref{thm:diamond}, the assumption that both $\mu,\nu$ are symmetric  in Theorem \ref{thm:p-cost} is essential for the diamond transport to minimize the QOT problem when $q\in (1,2]$. 
In some unreported numerical     results, we find that for asymmetric marginals,   (approximate) QOT minimizers are quite different from the diamond transport even for the simple case $q=2$.

For the cost function 
$c(x,y,x',y')=|(x-x')(y-y')|^q,~q\in [1,2]$, 
a maximizer of the QOT problem is given by the comonotone coupling $\pi_{\rm com}$ when $\mu=\nu$, as explained in Example \ref{ex:47}-(iii) (this also holds true for $q>0$). 
For this class of QOT problems, 
we know neither explicit maximizers when $\mu$ and $\nu$ are not identical 
nor explicit minimizers when $\mu$ and $\nu$ are not symmetric.

\section{Conclusion}\label{sec:concl}
The new framework of quadratic-form optimal transport (QOT) is proposed, with the key feature that the transport cost is linear in $\pi\otimes\pi$. Due to the possible non-convex structure, the QOT problem is difficult to solve numerically. 
We prove fundamental properties of QOT and highlight cases with explicit solutions, summarized in Table \ref{tab:example}. 
Compared to classic OT, QOT gives rise to two new and special optimal transport plans, the V-transport and the diamond transport. The latter is particularly interesting since it is not Monge, but serves as a universal minimizer of wide classes of QOT problems (Theorems \ref{thm:diamond} and \ref{thm:p-cost}), some of which are non-convex.

As a new framework, there are many unsolved problems on QOT. We briefly list some promising and important directions below. Details of these directions are further explained in Appendix \ref{sec:EC-disc}.

\begin{enumerate}[(i)]

\item In view of Brenier's theorem (\cite{brenier1987decomposition}), 
classic OT has a Monge solution under standard conditions.
We wonder whether a similar phenomenon exists for QOT. For instance, some transport plans supported on the union of the graphs of two maps, such as $\pi_{\mathrm{dia}}$ and $\pi^{\lambda}_{\mathrm{x}}$ are optimal for some QOT problems, but a general picture is not clear.  

\item It is worth exploring how our explicit QOT results, especially the diamond transport, can lead to solutions to QAP.


\item QOT may give rise to various applications to classic OT, especially through regularization that is different from the quadratically regularized OT discussed in Example \ref{ex:qf-reg}. 

    \item It remains open to solve explicitly the QOT minimizers for many simple cost functions, 
    such as that in Theorem \ref{thm:diamond} without symmetry assumption,   $|(x-x')(y-y')|^q$ for $q>1$, and 
    $\min\{|x-x'|,|y-y'|\}$.

\end{enumerate}

\section*{Acknowledgments}
We thank Jose Blanchet,  Job Boerma, Marcel Nutz, and Johannes Wiesel for helpful discussions,   Shengtong Zhang for proposing a simple proof of Theorem \ref{prop:pidia} in the case of uniform marginals,  and Facundo M\'{e}moli for  bringing up a series of relevant literature on the Gromov--Wasserstein distance. We also thank two referees for their helpful comments that improved this paper. 
R.~Wang acknowledges financial support from the Natural Sciences and Engineering Research Council of Canada (RGPIN-2024-03728, CRC-2022-00141). 

\section*{Declaration of interests}
The authors have no competing interests to declare that are relevant to the content of this article.

\bibliography{bib}
\bibliographystyle{abbrvnat}

  \appendix




\begin{center}
\LARGE \bf Appendices
\end{center}

\section{Quadratic programming formulation}
\label{app:QPF}

In the discrete case where $\mu$ and $\nu $ are supported on $N$ and $M$ points, respectively, QOT can be formulated by a quadratic program.
Denote by $\{x_1,\dots,x_N\}$ the support of $\mu$  
and by $\{y_1,\dots,y_M\}$ the support of $\nu$.
Let $\mu_i=\mu(\{x_i\})$  for $i\in [N]$
and $\nu_j=\nu(\{y_j\})$  for $j\in [M]$. 
Now the measure $\pi$ can be expressed by a matrix, and we write
$\pi=(\pi_{ij})_{i\in [N], j\in [M]}$.
QOT
can be written as the quadratic program 
\begin{align}
\label{eq:program}
\begin{aligned}
\mbox{to minimize~~} &\sum_{i,k\in [N],j,\ell\in [M]}c(x_i,y_j,x_k,y_\ell) \pi_{ij} \pi_{k\ell} \\
\mbox{over~~}&  \pi\in \R_+^{N\times M} 
\\ \mbox{subject to~~}& \sum_{j\in [M]}  \pi_{ij}=\mu_i  
\mbox{~~for all $i\in [N]$}\\& \sum_{i\in [N]} 
 \pi_{ij}=\nu_j \mbox{~~for all $j\in [M]$}.
 \end{aligned}
\end{align}
If one considers Monge QOT, there is the extra constraint that $\pi_{ij}\in\{0,\mu_i\}$, and the problem is not a quadratic program.

Further,
let us denote by  $\bdpi\in
\R^{NM} 
$  the vectorization of $\pi$, which has entries 
$$
(\bdpi)_{i(M-1)+j} = \pi_{ij},~~~~i\in [N];~j\in [M].
$$
Let 
$\boldsymbol \mu\in
\R^{N} 
$ and   $\boldsymbol \nu\in
\R^{M} 
$ be the vectorizations of  $\mu$ and $\nu$,
respectively, 
and let $\mathbf 1_n$ be the vector $(1,\dots,1)\in \R^n$. 
Moreover, let $C$ be the $NM\times NM$ matrix with entries  given by $$
C_{i(M-1)+j, k(M-1)+\ell}=
c(x_i,y_j,x_k,y_\ell),~~~~i,k\in [N];~j,\ell\in [M].$$ 
Then \eqref{eq:program}
has the following concise form
\begin{align}
\label{eq:program2}
\begin{aligned}
\mbox{to minimize~~} &  
\bdpi ^\top C \bdpi
\\
\mbox{over~~}&  \bdpi\in \R_+^{NM} 
\\ \mbox{subject to~~}& 
\pi\mathbf{1}_N  = \boldsymbol \mu \mbox{~~and~~} \pi^\top  \mathbf{1}_M = \boldsymbol \nu. 
 \end{aligned}
\end{align}
Note that the constraints in \eqref{eq:program2} are written in matrix form, which can also be written in vector form, but is less concise.


\section{Independent coupling is rarely a QOT minimizer}\label{sec:no ind}

In this section, we prove the following result, which states that similarly to classic OT, under mild conditions, the independent coupling cannot be a QOT optimizer. Recall the notation $\pi_{\mathrm{ind}}=\mu\otimes\nu$.

\begin{proposition}\label{prop: no ind}
Let $\mu\in\M(\X)$ and $\nu\in\M(\Y)$. 
Suppose that there exist continuous functions $c_X\in L^1(\mu)$ and $c_Y\in L^1(\nu)$ such that the cost function  $c$ is jointly continuous and satisfies
\begin{align}
    |c(x,y,x',y')|\leq c_X(x)+c_Y(y)+c_X(x')+c_Y(y')\label{eq:L1 cond}
\end{align}
pointwise. Then for the following statements:
\begin{enumerate}[(i)]
    \item every $\pi\in\Pi(\mu,\nu)$ is a minimizer of \eqref{eq:QOT cost};
    \item $\pi_{\mathrm{ind}}$ is a minimizer or maximizer of \eqref{eq:QOT cost};
    \item there exist functions $\varphi:\X\to\R$ and $\psi:\Y\to\R$ such that 
\begin{align}
    \tilde{c}(x,y):=\frac{1}{2}\iint \left( c(x,y,x',y')+c(x',y',x,y)\right) \d\mu(x')\,\d\nu(y')=\varphi(x)+\psi(y),\quad \pi_{\mathrm{ind}}\text{-a.e.},\label{eq:tilde}
\end{align}
\end{enumerate}
we have $(i)\implies (ii)\implies (iii)$. Moreover, if the cost function satisfies the conditions of Proposition \ref{prop:convex?}, (iii) implies that $\pi_{\rm ind}$ is a minimizer of \eqref{eq:QOT cost}, and hence (ii) and (iii) are equivalent.
    \end{proposition}

   \begin{proof} That $(i)\implies(ii)$ is trivial, so we prove $(ii)\implies(iii)$. 
Let $\pi\in\Pi(\mu,\nu)$ be arbitrary and suppose that $\pi_{\mathrm{ind}}$ is a minimizer. For $\delta\in[0,1]$, let $\pi_\delta=\delta\pi+(1-\delta)\pi_{\mathrm{ind}}\in \Pi(\mu,\nu)$. By optimality of $\pi_{\mathrm{ind}}$ and \eqref{eq:L1 cond}, for all $\delta\in[0,1]$,
\begin{align*}
    \iint c\,\d\pi_{\mathrm{ind}}\otimes \d\pi_{\mathrm{ind}}&\leq \iint c\,\d\pi_\delta\otimes\d\pi_\delta\\
    &=(1-2\delta) \iint c\,\d\pi_{\mathrm{ind}}\otimes \d\pi_{\mathrm{ind}}+\delta\iint c\,\d(\pi_{\mathrm{ind}}\otimes \d\pi_\delta+\d\pi_\delta\otimes\pi_{\mathrm{ind}})+O(\delta^2)
    .  
\end{align*}
Therefore, we must have, for $\tilde{c}$ defined in \eqref{eq:tilde}, 
\begin{align}
     \frac{1}{2}\int \tilde  c \, \d\pi_{\rm ind} =\iint c\,\d\pi_{\mathrm{ind}}\otimes \d\pi_{\mathrm{ind}}\leq \frac{1}{2}\iint c\,\d(\pi_{\mathrm{ind}}\otimes \d\pi_\delta+\d\pi_\delta\otimes\pi_{\mathrm{ind}}) 
= \frac{1}{2}\int \tilde  c \, \d\pi_\delta 
.\label{eq:1/2}
\end{align}
Hence, $\pi_{\mathrm{ind}}$ is a minimizer of the classic OT problem with cost function  $\tilde{c}$. By our assumptions and the dominated convergence theorem, $\tilde{c}$ is continuous. Therefore, classic OT duality (\citet[Theorem 5.10]{Villani:2009}) yields a $\tilde{c}$-cyclically monotone set $\Gamma$ and dual potentials $\varphi,\psi$ such that $\pi_{\mathrm{ind}}(\Gamma)=1$ and $\varphi(x)+\psi(y)=\tilde{c}(x,y)$. This implies \eqref{eq:tilde}. The case of $\pi_{\mathrm{ind}}$ being a maximizer can be similarly established.     

For the final statement, suppose that (iii) holds but $\pi$ has a strictly smaller quadratic-form transport cost than $\pi_{\rm ind}$. Denote by $\pi_\delta=\delta\pi+(1-\delta)\pi_{\mathrm{ind}}$. By the convexity of the quadratic-form transport cost and a similar argument leading to \eqref{eq:1/2}, it holds that $\int\tilde{c}\,\d\pi_{\rm ind}>\int\tilde{c}\,\d\pi_\delta$ for $\delta>0$ small enough. This means that $\pi_{\rm ind}$ is not a minimizer of the classic OT problem with cost $\tilde{c}$, contradicting the separability assumption (iii).
    \end{proof}

    \begin{example}\label{ex...}
        Suppose that $\mu,\nu$ are discrete with the cost function given by $c(x,y,x',y')=\bone_{\{x=x',y=y'\}}/(\mu(\{x\})\nu(\{y\}))$. This is precisely the regularization term in quadratically regularized OT as discussed in Example \ref{ex:qf-reg}; see \eqref{eq:qreg}. By Jensen's inequality, the unique minimizer of this QOT problem is given by $\pi_{\rm ind}$. In this case, the cost $\tilde{c}$ in \eqref{eq:tilde} is constant one.
    \end{example}

    Condition \eqref{eq:tilde} can often be checked explicitly (which often does not hold for common cost functions) and thus offers a neat necessary condition for $\pi_{\mathrm{ind}}$ to be a minimizer and for the QOT problem to be trivial. As a sanity check, for Example \ref{ex:intro}, \eqref{eq:tilde} can be easily verified as $\tilde{c}$ is constant on $\{0,1\}^2$. The same example (as well as Example \ref{ex...}) also shows that (ii) does not imply (i) in general. We next provide two counter-examples that satisfy (iii) but not (ii). 

    \begin{example}
Let $\mu,\nu$ be the standard normal distribution, and the cost function $c$ given by
$
c(x,y,x',y')=((xy)^2-1)(x'y').
$
It is easy to verify 
 $\tilde c=0$.
 If $\pi\in \Pi(\mu,\nu)$ is the joint normal distribution 
 with correlation coefficient $\rho\in [-1,1]$, 
 then 
 $\iint c \, \d\pi \otimes \d \pi = 2\rho^3.$ 
 Clearly, this transport cost is neither maximized nor minimized by the independent coupling $(\rho=0)$. 
    \end{example}

    \begin{example}
        Consider $\mu$ uniform on $\X=\{0,1\}$ and $\nu$ uniform on $\Y=\{0,1,2,3\}$, with cost function $c$ given by 
        $c(1,0,1,0)=1$, $c(0,0,0,0)=1$, $c(1,1,1,1)=-1$, $c(0,1,0,1)=-1$, 
        $c(1,2,1,2)=1$,
        $c(0,2,0,2)=1$, and zero otherwise.
        Denote the transition probabilities from $0\in\X$ to $0,1,2,3\in\Y$ respectively by $p,q,r,1-(p+q+r)$. It follows that the transition probabilities from $1\in\X$ are $1/2-p,1/2-q,1/2-r,p+q+r-1/2$. So, the total transport cost is
        $$\Big(\frac{1}{2}-p\Big)^2+p^2-\Big(\frac{1}{2}-q\Big)^2+q^2+\Big(\frac{1}{2}-r\Big)^2+r^2= 2\bigg(\Big(\frac{1}{4}-p\Big)^2-\Big(\frac{1}{4}-q\Big)^2+\Big(\frac{1}{4}-r\Big)^2\bigg)+\frac{1}{8}.$$
        Note that the set of all couplings $\Pi(\mu,\nu)$ can be parameterized by $\{(p,q,r)\in[0,1/2]^3:1/2\leq p+q+r\leq 1\}$. 
        The minimizers are then given by $(1/4,0,1/4)$ and $(1/4,1/2,1/4)$, and the maximizers are given by $(0,1/4,1/2)$ and $(1/2,1/4,0)$. None of them is the independent coupling, which is given by $(1/4,1/4,1/4)$. 
        On the other hand, straightforward calculation shows that $\tilde{c}(x,y)$ is a function of $y$ only. For instance, by symmetry of $c$,
        $$\tilde{c}(x,0)=\sum_{x'\in\X}\sum_{y'\in\Y}c(x,0,x',y')\mu(\{x'\})\nu(\{y'\})=\begin{cases}
            c(0,0,0,0)\mu(\{0\})\nu(\{0\})=1/8&\text{ if }x=0;\\
            c(1,0,1,0)\mu(\{1\})\nu(\{0\})=1/8 &\text{ if }x=1.
        \end{cases}$$
        
    \end{example}

\section{Linear-exponential distance cost functions}
\label{sec:dispersion}

\subsection{Basic facts}

All cost functions in Section \ref{sec:diamond} are symmetric in $|x-x'|$ and $|y-y'|$. 
In this section, we consider a special class of type-XX cost function (a sub-class of the one treated in Theorem \ref{prop:cocounter}) that is not symmetric in $|x-x'|$ and $|y-y'|$,
which we call the class of \emph{linear-exponential distance} cost functions, defined by 
\begin{align}
    c_\gamma(x,y,x',y'):=|y-y'|e^{-\gamma|x-x'|},\quad
    \mbox{$\gamma>0$}.\label{eq:Dc}
\end{align}
 In probabilistic terms, $\iint c_\gamma\,\d\pi\otimes\d\pi=\E\big[|Y-Y'|e^{-\gamma|X-X'|}\big]$ for $(X,Y),(X',Y')\lawis \pi$ iid. The intuition is that the cost function \eqref{eq:Dc} measures the difference between $Y$ and $Y'$ when $X$ and $X'$ are close,
and the parameter $\gamma$ controls 
how the distance between $X$ and $X'$
is discounted. The QOT problem is then formulated as 
   \begin{align}
    \begin{split}
         \mbox{to minimize}\quad&\iint |y-y'|e^{-\gamma|x-x'|}\,\d\pi(x,y)\d \pi (x',y')\\
        \mbox{subject to}\quad &\pi\in\Pi(\mu,\nu).
    \end{split}\label{eq:D}
    \end{align}

We summarize the results for this class of QOT in the following, where we observe that the minimizers and maximizers lead to very different mathematical structures. 
\begin{enumerate}[(i)]
    \item Assume $\mu \in\P_1(\R)$ and $\nu$ is an increasing (resp.~decreasing) location-scale transform of $\mu$. By checking the conditions in Theorem \ref{prop:cocounter}, the minimizers of \eqref{eq:D}
    are (a) the comonotone (resp.~antimonotone) coupling when $\mu$ is asymmetric, and 
    (b) the comonotone and antimonotone couplings  when $\mu$ is symmetric. 
\item Assume $\mu\in\P_2(\R)$ and $\nu\in\P_{2+\delta}(\R)$ for some $\delta>0$. As $\gamma \to 0^+$, the maximizer of   \eqref{eq:D} converges weakly to the diamond transport.
\item Assume $\mu\in\M(\R)$ and $\nu\in\P_1(\R)$. As $\gamma \to \infty$, the unique maximizer of the limit of \eqref{eq:D} is the independent coupling. The set of minimizers of the limit of \eqref{eq:D} is given by the set of Monge maps $\T(\mu,\nu)$.
\end{enumerate}

For the minimizers, we only obtain marginals in the same location-scale class as in the first item.
For the maximizers, we do not need to assume identical marginals, but we only have asymptotic results as in the second and the third items.
For arbitrary $\mu,\nu$ and fixed $\gamma>0$, we do not know either the   minimizer or the  maximizer for the QOT problem in general.

The first item above is rigorously presented in the following simple proposition.

\begin{proposition}\label{prop:dgamma basic}
  Suppose that  $\gamma>0$, $\mu\in\P_1(\R)$, and $\nu$ is an increasing (resp.~decreasing) location-scale transform of $\mu$. If $\mu$ is symmetric, the set of all minimizers of \eqref{eq:D} is given by the comonotone and antimonotone couplings.    If $\mu$ is asymmetric, the comonotone (resp.~antimonotone) coupling is the unique minimizer.
\end{proposition}

\begin{proof}
This follows directly from Theorem \ref{prop:cocounter} applied with $h(s,t)=se^{-\gamma t}$. To check the conditions, note that the function $h$ is strictly submodular as the product of a strictly increasing function in $s$ and a strictly decreasing function in $t$. In addition, since $\int |x|\d\mu(x)<\infty$, the independent coupling has a finite quadratic-form transport cost, and hence so does the comonotone and antimonotone couplings. 
\end{proof}

In the remainder of this section, we analyze the limit behavior of the optimizers of \eqref{eq:D} as $\gamma\to 0^+$ and as $\gamma\to\infty$, corresponding to the second and third items above.

\subsection{The first limit case}

Observe that for any $\gamma>0$, a maximizer of \eqref{eq:D} also minimizes
$$\iint|y-y'|\Big(\frac{1-e^{-\gamma|x-x'|}}{\gamma}\Big)\,\d\pi(x,y)\,\d\pi(x',y')$$
over $\pi\in\Pi(\mu,\nu)$. 
Formally, as $\gamma\to 0^+$, we arrive at the limit optimization problem
\begin{align*}
        \begin{split}
            \mbox{to minimize}\quad&\iint|y-y'||x-x'|\,\d\pi(x,y)\,\d\pi(x',y')\\
        \mbox{subject to}\quad
        & \pi\in\Pi(\mu,\nu).
        \end{split}
    \end{align*}
We have shown in Theorem \ref{prop:pidia} above that if $\mu,\nu\in\P_1(\R)$, the unique minimizer is given by the diamond transport $\pi_{\mathrm{dia}}$ in Definition \ref{def:diamond}. 
For each $\gamma>0$, let $\pi^\gamma$ be a maximizer of \eqref{eq:D}.

\begin{proposition}\label{prop:convergence gamma to 0}
Let $\mu\in\P_2(\R)$ and $\nu\in\P_{2+\delta}(\R)$ for some $\delta>0$.  Then $\lim_{\gamma\to 0^+}\pi^\gamma=\pi_{\mathrm{dia}}$ weakly.
\end{proposition}

\begin{proof}
By the definition \eqref{eq:D} and Fact \ref{rem:marginals}, for each $\gamma>0$, $\pi^\gamma$ is also a minimizer of
\begin{align}
    \iint|y-y'|\Big(\frac{1-e^{-\gamma|x-x'|}}{\gamma}\Big)\,\d\pi(x,y)\,\d\pi(x',y').\label{eq:gamma pi}
\end{align}
Without loss of generality, $\delta\in(0,1)$. Observe the elementary inequality $|u+e^{-u}-1|\leq u^{1+\delta}$ for $u\geq 0$.\footnote{To see this, observe that $u+e^{-u}-1\geq 0$. Let $f(u)=u^{1+\delta}+1-u-e^{-u}$. Clearly, $f(u)\geq 0$ for $u\geq 1$. For $u\in(0,1)$, $f'(u)=(1+\delta)u^\delta+e^{-u}-1\geq (1+\delta)u^\delta-u\geq 0$. Therefore, $f(u)\geq 0$ for $u\in[0,1]$.} It follows that uniformly in $\pi\in\Pi(\mu,\nu)$,
\begin{align*}
   & \bigg|\iint|y-y'|\Big(\frac{1-e^{-\gamma|x-x'|}}{\gamma}\Big)\,\d\pi(x,y)\,\d\pi(x',y')-\iint|y-y'||x-x'|\,\d\pi(x,y)\,\d\pi(x',y')\bigg|\\
   &\leq \gamma^{\delta/2} \iint|y-y'||x-x'|^{1+\delta/2}\,\d\pi(x,y)\,\d\pi(x',y')\\
   &\leq \gamma^{\delta/2}\bigg(\iint |y-y'|^2\d\nu(y)\,\d\nu(y')\bigg)^{1/2}\bigg(\iint |x-x'|^{2+\delta}\d\mu(x)\,\d\mu(x')\bigg)^{1/2},
\end{align*}
where the right-hand side does not depend on $\pi$. Therefore, the functional \eqref{eq:gamma pi} converges uniformly to $\iint |y-y'||x-x'|\,\d\pi(x,y)\,\d\pi(x',y')$ as $\gamma\to 0^+$. By Theorem \ref{prop:pidia}, $\iint |y-y'||x-x'|\,\d\pi(x,y)\,\d\pi(x',y')$ is uniquely minimized by $\pi_{\mathrm{dia}}$. Hence, $\pi^\gamma\to\pi_{\mathrm{dia}}$ weakly.
\end{proof}

\subsection{The second limit case: Weak OT and measures of association}
Next, we study the limit behavior as $\gamma\to\infty$. 
Consider the scaled version of \eqref{eq:D}:
$$\frac{\gamma}{2}\iint c_\gamma\,\d\pi\otimes\d\pi=\iint|y-y'|\frac{\gamma}{2} e^{-\gamma|x-x'|}\d\pi(x,y)\,\d\pi(x',y').$$
As $\gamma\to\infty$, the double integral has a formal limit of $\E[|Y-Y''|]$, where $Y,Y''$ are conditionally iid~on $X$ and $(X,Y)\lawis\pi$, which is closely connected to measures of association studied by \cite{deb2020measuring,deb2024distribution}.

The limit is verified for well-behaved couplings $\pi$ in Proposition \ref{lemma:convergence} below. Stated in probabilistic terms, the following optimization problem arises as $\gamma\to\infty$:
\begin{align}
        \begin{split}
            \mbox{maximize}\quad&\E[|Y-Y''|]\\
        \mbox{subject to}\quad &Y,\,Y''\mbox{ are conditionally independent given }X;\\
        &(X,Y),\,(X,Y'')\lawis \pi;\\
        &\pi\in\Pi(\mu,\nu).
        \end{split}\label{eq:limit problem}
    \end{align}
       This problem does not belong to our QOT framework but is a weak optimal transport problem. For $\pi\in\Pi(\mu,\nu)$, let $\kappa=\{\kappa_x\}_{x\in\X}$ be a regular disintegration with respect to the first marginal, and we write $\pi=\mu\otimes\kappa$. Given a cost function $c:\X\times\M(\Y)\to\R$, the \textit{weak optimal transport problem} is 
       \begin{align*}
        \mbox{to minimize}\quad&\int c(x,\kappa_x)\,\d \mu(x)\\
        \mbox{subject to}\quad &\pi\in\Pi(\mu,\nu).
    \end{align*}
    We refer to \cite{backhoff2019existence} and \cite{gozlan2017kantorovich} for  thorough treatments on this topic.

\begin{proposition}\label{lemma:convergence}
 Let  $\mu\in\M(\R)$ and $\nu\in\P_1(\R)$.   Suppose that $\pi=\mu\otimes\kappa\in\Pi(\mu,\nu)$ satisfies the following: either $\pi$ is absolutely continuous with respect to the Lebesgue measure on $\R^2$, or $\kappa_x$ is continuous in $x$ in the weak topology. Assume further that there exist constants $C,L>0$ such that for $\mu$-a.e.~$x$, $\int y\kappa_x(\d y)\leq C$ and $\mu$ has a continuous density $f\le L$ with respect to the Lebesgue measure. Then 
    $$\iint|y-y'|\frac{\gamma}{2} e^{-\gamma|x-x'|}\d\pi(x,y)\,\d\pi(x',y')\to \E[|Y-Y''|],$$
    where $Y,Y''$ are conditionally iid given $X$.
\end{proposition}

\begin{proof}
Define $g(x,y):=\int |y-y'|\kappa_{x}(\d y')$. We first claim that under our assumptions,
\begin{align}
    \lim_{\gamma\to \infty}\int\frac{\gamma}{2}e^{-\gamma|u|}g(x+u,y)f(x+u)\,\d u= g(x,y)f(x) \quad\pi\text{-a.e.}\label{eq:limit}
\end{align} 
Indeed, if $\kappa_x$ is continuous in $x$, $g(x,y)$ would be jointly continuous in $(x,y)$, and hence \eqref{eq:limit} is valid. On the other hand, by Lebesgue's differentiation theorem, \eqref{eq:limit} holds for $x$-a.e.~for fixed $y$, and hence for Lebesgue-a.e.~$(x,y)$. If $\pi$ is absolutely continuous, \eqref{eq:limit} also holds for $\pi$-a.e.~$(x,y)$.    

Using $\p(Y'\le x \mid X')=\p( Y'\le x\mid X,X')$ almost surely~for all $x\in \R$, we have by first conditioning on $(X,Y)$ and then conditioning on $X'$ that
\begin{align}
    \begin{split}
        \E\Big[\frac{\gamma}{2}|Y-Y'|e^{-\gamma|X-X'|}\Big]&=\int \frac{\gamma}{2}\E\big[|y-Y'|e^{-\gamma|x-X'|}\big]\d\pi(x,y)\\
    &=\iint \frac{\gamma}{2}g(x',y)e^{-\gamma|x-x'|}f(x')\,\d x'\d\pi(x,y)
    \end{split}\label{eq:e2}
\end{align}
Next, we apply \eqref{eq:limit} and the dominated convergence theorem to show that as $\gamma\to\infty$,
\begin{align}
    \iint\frac{\gamma}{2}g(x',y)e^{-\gamma|x-x'|}f(x')\,\d x'\d\pi(x,y)\to \int g(x,y)f(x)\,\d\pi(x,y).\label{eq:e5}
\end{align}
To see this, it remains to verify
\begin{align}
    \int\sup_{\gamma\geq 0}\int\frac{\gamma}{2}g(x',y)e^{-\gamma|x-x'|}f(x')\,\d x'\d\pi(x,y)<\infty.\label{eq:e3}
\end{align}
By our assumption and the triangle inequality, $g(x,y)\leq |y|+C$ and $f(x')\leq L$. 
It follows that uniformly in $\gamma\geq 0$,
\begin{align*}
    \int\frac{\gamma}{2}g(x',y)e^{-\gamma|x-x'|}f(x')\,\d x'&\leq \frac{\gamma L}{2}\int (|y|+C)e^{-\gamma|x-x'|}\d x' \leq L(C+|y|).
\end{align*}
Therefore,
$$\int\sup_{\gamma\geq 0}\int\frac{\gamma}{2}g(x',y)e^{-\gamma|x-x'|}f(x')\,\d x'\d\pi(x,y)\leq \int L(C+|y|)\,\d\pi(x,y)<\infty.$$
This proves \eqref{eq:e3}. 
The proof is then complete, by \eqref{eq:e2}, \eqref{eq:e5}, and the observation that
$$\int g(x,y)f(x)\,\d\pi(x,y)=\E[|Y-Y''|],$$
where $Y,Y''$ are conditionally iid given $X$.
\end{proof}

In Proposition \ref{prop:gamma to infty} below, we explicitly solve \eqref{eq:limit problem}. The same problem is studied as a special case of Proposition 1.1 of \cite{deb2020measuring} as a measure of association of $(X,Y)$, given by 
$$
\eta (X,Y)= 1-\frac{\E[|Y-Y''|]}{\E[|Y-Y'|]} \in [0,1],
$$
where $Y$ and $Y'$ are iid and $Y$ and $Y''$ are conditionally iid given $X$. 
Under some additional assumptions on $(X,Y)$, \cite{deb2020measuring} showed that if $X$ and $Y$ are non-degenerate, then $\eta(X,Y)=0$ if and only if $X$ and $Y$ are independent,
and $\eta(X,Y)=1$ if and only if $Y$ is a measurable function of $X$. 
Our next result, with a self-contained proof, implies the above conclusion on $\eta$. It  assumes only the first moment condition  on $Y$, much weaker than the conditions in \cite{deb2020measuring}.
    
\begin{proposition}\label{prop:gamma to infty}
Suppose that $\mu\in\M(\R)$ and $\nu\in\P_1(\R)$.     The unique maximizer $\pi$ to \eqref{eq:limit problem} is given by the independent coupling $\pi_{\mathrm{ind}}$. The set of minimizers of \eqref{eq:limit problem} is given by the set of Monge maps $\T(\mu,\nu)$. 
\end{proposition}

\begin{proof}  
We first analyze the maximizer of \eqref{eq:limit problem}. Since $\nu\in\P_1(\R)$, the independent coupling yields a finite transport cost.
Let $X,U,V$ be independent, with
$X\lawis \mu$ and $U,V\lawis\mathrm{U}$.  
Suppose that $Y$ and $Y'$
are conditionally iid~on $X$, and $Y\lawis \nu$.
For any coupling $(X,Y)$, we note that $(Y,Y',X)\laweq (f(X,U),f(X,V),X)$ where $u\mapsto f(x,u)$ is (a regular version of) the conditional quantile function of $Y$  on $X=x$. 
Write $g(t)=\E[f(X,t)]$ for $t\in [0,1]$, which implies
$g(U)=\E[f(X,U)\mid U]$. 
It follows that \begin{align*}
\E[|Y-Y'|] &=\E[f(X,U)\vee f(X,V)] - \E[f(X,U)\wedge f(X,V)] 
\\&=  \E[f(X,U\vee V)] - \E[f(X,U\wedge V)] 
\\&=\int_0^1 g(t)  \d t^2 - \int_0^1 g(t)\,\d (2t- t^2)
\\&=  2\int_0^1 g(t) (2t -1)  \d t . 
 \end{align*} 
For two random variables,  we write the convex order relation
$Z\le_{\rm cx} W$  if 
$\E[h(Z)]\le \E[h(W)]$ for all  convex functions $h$ such that the two expectations are well-defined.     
Note that $g(U) \le_{\rm cx} f(X,U) \laweq Y$. 
Moreover, $g(U)\laweq Y$ when $X$ and $Y$ are independent. 
By \citet[Theorem 4.5]{furman2017gini}, the functional $X\mapsto \int _0^1 (2t -1) Q_\mu(t)\,\d t$ is strictly increasing in convex order, where $X\lawis \mu$.
Therefore, 
$\E[|Y-Y'|] $ is maximized if and only if $g(U)\laweq Y$.
Therefore, for the maximizer $(X,Y)$,
$\E[f(X,U)\mid U]=g(U)=f(X,U)$ holds true, implying that $X$ and $f(X,U)$ are independent as $X$ and $U$ are independent. This shows that the independent coupling $\pi_{\mathrm{ind}}$ is the unique maximizer of $\E[|Y-Y'|]$.

Next, we derive the set of minimizers. If the coupling $(X,Y)$ is induced by a Monge map $Y=f(X)$ for some measurable $f$, the objective is $\E[|Y-Y'|]=\E[|f(X)-f(X)|]=0$. Conversely, write $\pi=\mu\otimes\kappa$.
If $ \E[|Y-Y'|]=0$, then $ \E[|Y-Y'|\mid X]=0$ almost surely, 
implying $\mu  (\{x:\kappa_x\text{ is  degenerate}\} )=1$,  proving that $(X,Y)$ is   Monge. 
\end{proof}

The upshot of the above results is that, although Proposition \ref{prop: no ind} implies that the independent coupling is never a minimizer for \eqref{eq:D} with $\gamma>0$, we expect that the maximizers $\pi^\gamma$ behave like the independent coupling as $\gamma\to\infty$. 
However, we do not have a proof to guarantee that the maximizer $\pi^\gamma$ of \eqref{eq:D} converges to  $\pi_{\rm ind}$. 



\subsection{Numerical approximations for the optimizers}\label{sec:simulations}

We next present some numerical approximation for QOT optimizers. The goal here is to understand how the QOT minimizers and maximizers  for the linear-exponential cost function behave when we cannot compute them explicitly (recall that, for minimizers, we need $\mu,\nu$ to be in the same location-scale family, and for maximizers,  we only have some limiting results). 

The QOT problems with cost functions $  c_\gamma$ and $-c_\gamma$ are not convex, indicating that exact solutions may be difficult to compute numerically,   and hence we apply heuristic local search algorithms to solve for an optimal transport map. 
More precisely, we apply the metaheuristic improvement method with pair exchange neighborhood (see Section 3.2 of \cite{cela2013quadratic}, or Section 8.2.3 of \cite{burkard2012assignment}) to a discretized version of the maximization problem with cost function \eqref{eq:Dc}, which is a quadratic assignment problem. The discretization procedure is justified by Proposition \ref{prop:stability}. The resulting matching may approximate the minimizers and maximizers of 
$\iint c_\gamma\,\d\pi\otimes\d\pi$. 

  The minimizer  of $\iint c_\gamma\,\d\pi\otimes\d\pi$  is known to be the comonotone coupling (Proposition \ref{prop:dgamma basic}) when $\mu,\nu$ are in the same location-scale family, and hence we choose uniform  and normal marginals, that is, $\mu=\mathrm{U}(0,1)$
and $\nu=\mathrm{N}(0,1)$.
In Figure \ref{fig:lambda}, we report 
the approximate minimizers (normalized by their ranks) with 
 parameters $\gamma\in\{0.3,\,2,\,6\}$, obtained from the numerical scheme above.
Each of them has an interesting   ``\reflectbox{$\lambda$}-shaped''  support,   clearly different from the comonotone coupling, or any other explicit coupling that we studied.

\begin{figure}
    \centering
    \begin{subfigure}{0.315\textwidth}
    \includegraphics[width=1\linewidth]{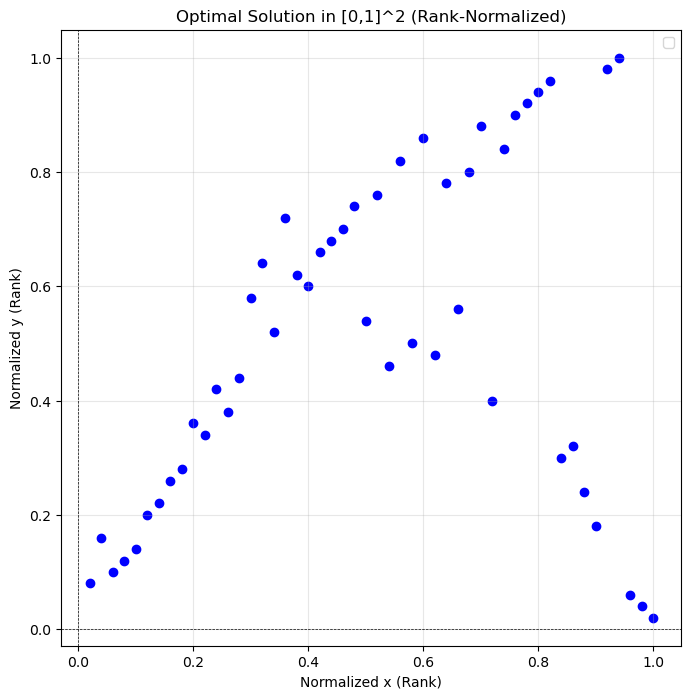}
    \caption{$\gamma=0.3$}
\end{subfigure}
\begin{subfigure}{0.315\textwidth}
\centering \includegraphics[width=1\linewidth]{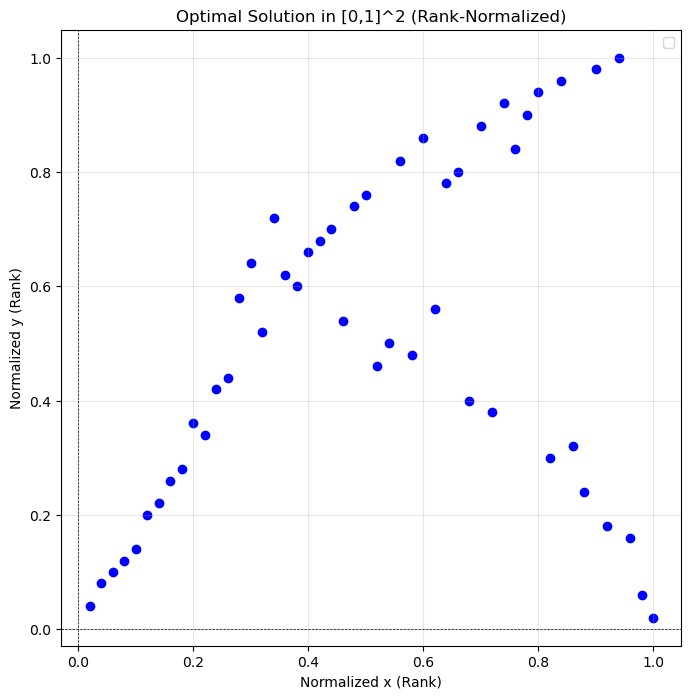}
\caption{$\gamma=2$}
\end{subfigure}
\begin{subfigure}{0.315\textwidth}
\centering \includegraphics[width=1\linewidth]{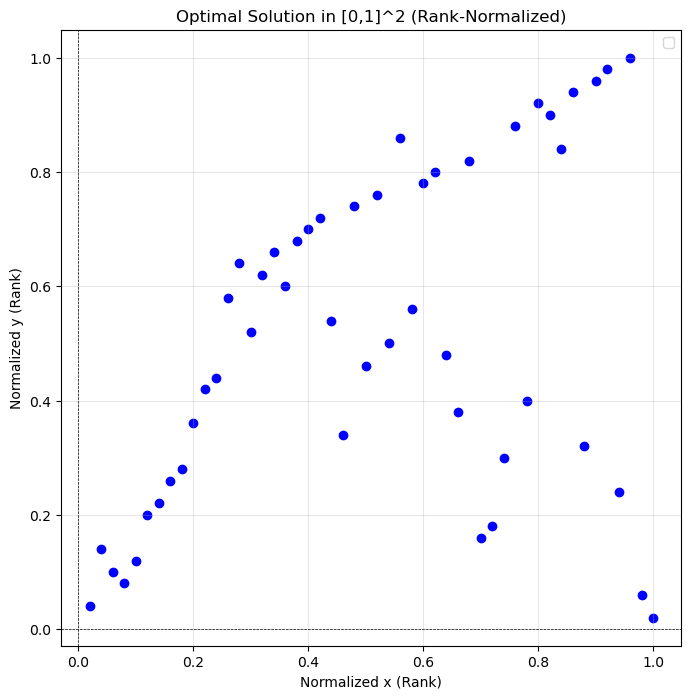}
\caption{$\gamma=6$}
\end{subfigure}
    \caption{Optimal coupling with quadratic-form cost function $c_\gamma(x,y,x',y')=|y-y'|e^{-\gamma|x-x'|}$ for $\gamma\in\{0.3,2,6\}$, where $\mu$ 
    and $\nu$ are the empirical measures of $50$ simulated points from $\mathrm{U}(0,1)$ and   from $\mathrm{N}(0,1)$, respectively. The optimal coupling is reported in   ranks. The support appears close to a \reflectbox{$\lambda$}-shape. The numerical procedure is based on the heuristic improvement method with pair exchange neighborhood with $500$ iterations, initiated from the comonotone transport $\pi_{\mathrm{com}}$.}
    \label{fig:lambda}
\end{figure} 

For the maximizers of $\iint c_\gamma\,\d\pi\otimes\d\pi$, we do not know explicit forms even for the case $\mu=\nu=\mathrm{U}(0,1)$, so we consider these marginals in the numerical scheme.
In Figure \ref{fig:dispersion maximizer}, we report  the 
 approximate maximizers with 
 parameters $\gamma\in\{0.3,\,2,\,6\}$. Maximizers for smaller $\gamma$ appear closer to $\pi_{\mathrm{dia}}$ (thus reassuring Proposition \ref{prop:convergence gamma to 0}), and for larger $\gamma$ appear closer to $\pi_{\mathrm{ind}}$  (thus reassuring Proposition \ref{prop:gamma to infty}). The support of the optimizer seems to be contained in a certain symmetric convex shape $E_\gamma$ in $[0,1]^2$. As $\gamma$ increases, the support expands, and less mass is concentrated near the boundary of $E_\gamma$ but more mass in the interior of $E_\gamma$. 
 
\begin{figure}[t]
\centering
\begin{subfigure}{0.315\textwidth}
    \includegraphics[width=1\linewidth]{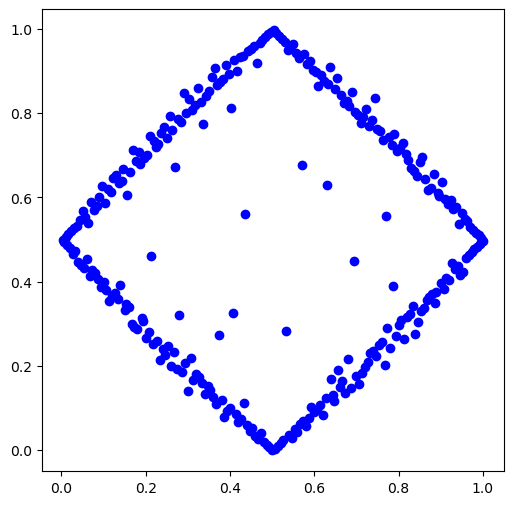}
    \caption{$\gamma=0.3$}
\end{subfigure}
\begin{subfigure}{0.315\textwidth}
\centering \includegraphics[width=1\linewidth]{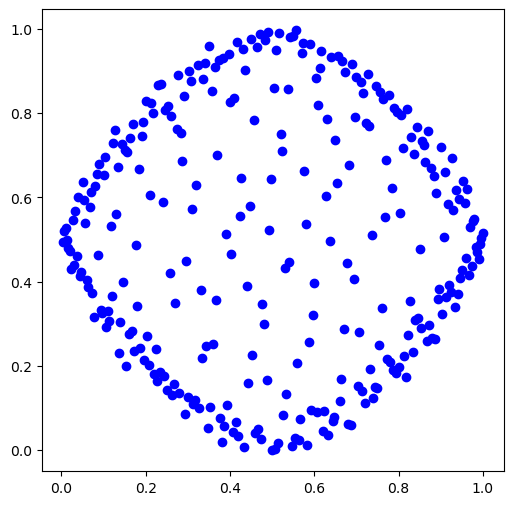}
\caption{$\gamma=2$}
\end{subfigure}
\begin{subfigure}{0.315\textwidth}
\centering \includegraphics[width=1\linewidth]{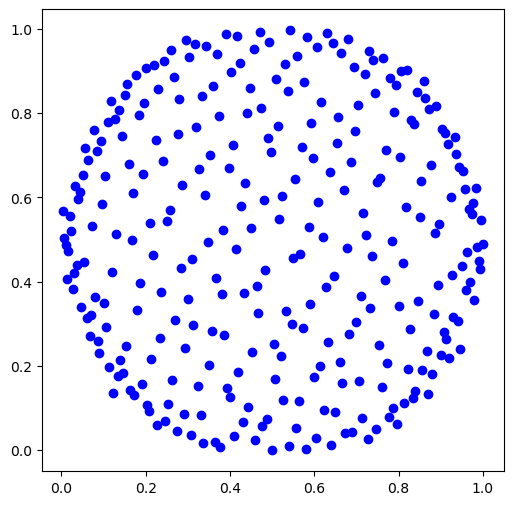}
\caption{$\gamma=6$}
\end{subfigure}
\caption{Plots of the QOT maximizers with cost function $c_\gamma(x,y,x',y')=|y-y'|e^{-\gamma|x-x'|}$ for $\mu,\nu$ both uniformly distributed on 300 equally spaced points on $[0,1]$ with various parameters $\gamma$. The numerical procedure is performed using the heuristic improvement method with pair exchange neighborhood with $10^4$ iterations, initiated from the (discretized version of the) diamond transport $\pi_{\mathrm{dia}}$. }
\label{fig:dispersion maximizer} 
\end{figure}

\section{Extended discussions on unsolved questions}
\label{sec:EC-disc}

 We provide details for the promising directions and unsolved problems outlined in Section \ref{sec:concl}.

\begin{enumerate}[(i)]

\item Brenier's theorem in classic OT states that if $\mu,\nu\in\M(\R^d)$, $\mu$ is absolutely continuous, and the cost is given by the squared Euclidean distance $\|x-y\|^2$, then the (unique) transport plan is Monge and induced by the gradient of a convex function (\cite{brenier1987decomposition}; see also \citet[Theorem 1.17]{Santambrogio:2015} for a more general version). The analogous question in the QOT context remains very challenging. Recent studies on the (2,2)-GW cost \eqref{eq:GWpq} suggest the existence of optimal 2-maps (transport plans supported on the union of the graphs of two maps, such as $\pi_{\mathrm{dia}}$, $\pi^\lambda_{\mathrm{x}}$, and $\pi_{\rm v}$ with $x,y$ flipped) and the non-existence of Monge minimizers under certain assumptions including absolute continuity of $\mu$ (\citet[Theorem 3.6]{dumont2024existence}). Our closed-form results (Propositions \ref{ex:quadratic cost} and \ref{ex:quadratic cost-f}, and Theorems \ref{def:wedge}, \ref{thm:diamond}, and \ref{thm:p-cost}) provide evidence that the existence of optimal 2-maps might be a universal phenomenon for many QOT problems (but not all of them, in view of Figure \ref{fig:dispersion maximizer}). This phenomenon is also present in many other extensions or special cases of classic OT (mostly on the real line) such as the martingale optimal transport (\citet[Corollary 1.6]{beiglbock2016problem}), the directional optimal transport (\citet[Corollary 2.9]{nutz2022directional}), and OT with concave costs (\citet[Theorem 6.4]{Gangbo:1996}).

\item Recent works on explicit solutions of QAP (\cite{burkard2012assignment,cela2018new}) may hint at certain cost structures leading to further closed-form minimizers of QOT. On the other hand, the analysis of the cost function $c(x,y,x',y')=|(x-x')(y-y')|^q,~q\in[1,2]$ seems reminiscent in the QAP literature, and hence our results in Section \ref{sec:diamond} may potentially inspire new explicitly solvable cases in QAP. In particular, we expect that solving for \eqref{eq:motivation1} above (equivalent to the case $q=2$) in the Monge setting may leverage on tools in the QAP literature, where the minimizer is a discrete approximation to the diamond transport in a suitable sense (since the transport cost is continuous in the weak topology).

\item We anticipate that our work will inspire various applications of QOT to classic OT. For example, building on the convex QOT cost functions introduced in Section \ref{sec:convex}, a theory of (convex) quadratic-form regularized OT can be developed. Unlike the quadratically regularized OT discussed in Example \ref{ex:qf-reg}, the quadratic-form approach offers a rich variety of parameterized regularizer classes (Example \ref{ex:dia}) and does not require the solution to be absolutely continuous with respect to the independent coupling.
We expect that the quadratic-form regularized OT generally also leads to sparse (or even singular) couplings.

    \item There are many simple cost functions for which we do not have an explicit solution to the corresponding QOT problem. We list a few examples below.
    \begin{enumerate}[(a)]
        \item 
  
We wonder whether Theorem \ref{thm:diamond} extends to marginal distributions that are not symmetric. We conjecture that some ``diamond-type'' coupling is the minimizer of the corresponding QOT problem. Such a coupling is a combination of four comonotone and antitone pieces, and it is numerically supported by Figure \ref{fig:asymmetric diamond}.  
 

\begin{figure}[t]

\begin{subfigure}{0.464\textwidth}
\centering \includegraphics[width=1\linewidth]{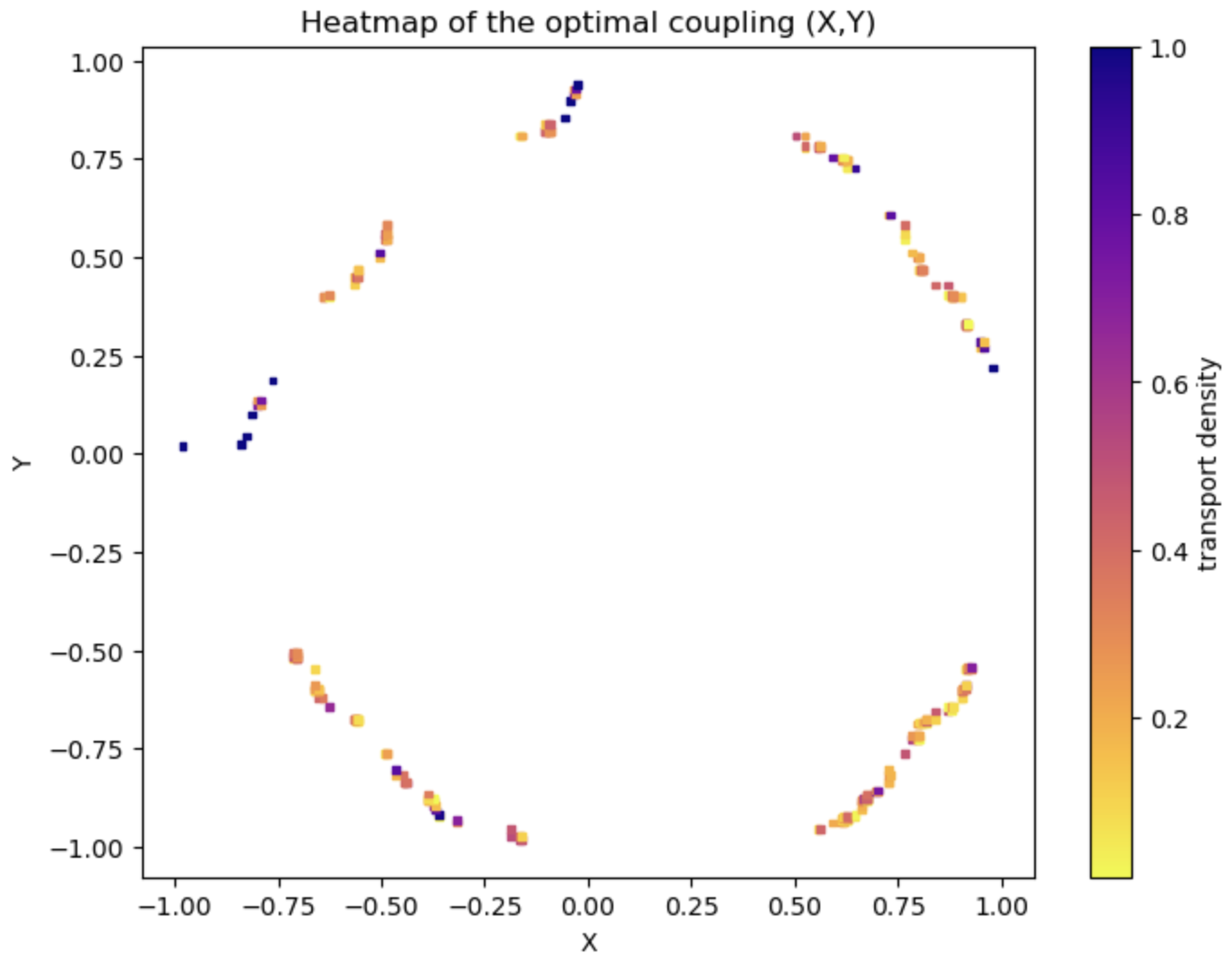}
\caption{the optimal coupling $(X,Y)$}
\end{subfigure}
\begin{subfigure}{0.45\textwidth}
\centering \includegraphics[width=1\linewidth]{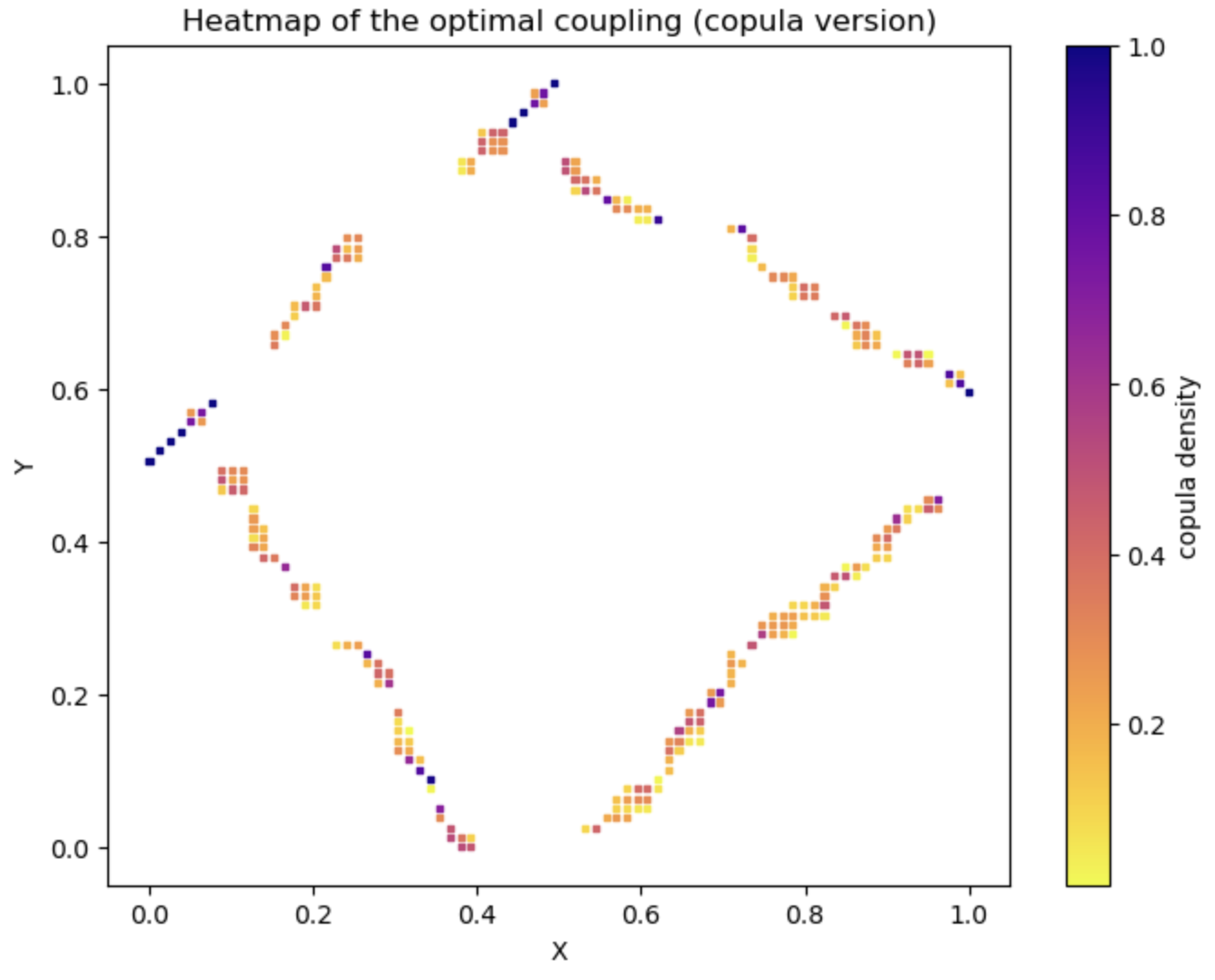}
\caption{copula of the optimal coupling $(X,Y)$}
\end{subfigure}
\caption{Plots of the quadratic-form optimal coupling $(X,Y)$ and its copula version, with cost function $e^{-((x-x')^2+(y-y')^2)/2}$ and marginals $\mu$ uniformly distributed on $[-1,0]\cup[1/2,1]$ and $\nu$ uniform on $[-1,-1/2]\cup[0,1]$, both approximated by $80$ iid samples. Since the QOT problem is convex (Theorem \ref{thm:diamond}), we apply the OSQP solver to find the optimal transport plan, illustrated with the heatmaps. The copula remains of diamond shape, but differs from the diamond copula which is perfectly symmetric.}
\label{fig:asymmetric diamond} 
\end{figure}
 


\item In Theorem \ref{thm:p-cost}, we explicitly solved a class of QOT problems with cost function $|(x-x')(y-y')|^q,\,1<q\leq 2$ by realizing it as a limit of other solvable classes. We conjecture that the moment condition can be relaxed to $\mu,\nu\in\P_q(\R)$ and the minimizer is unique. The case $q>2$ also deserves future study, as it is equivalent to maximization of the $(2,q)$-GW transport cost in \eqref{eq:GWpq}.

\item In addition to the results we obtained and the conjectures above, many other cost functions may yield explicit optimizers of the QOT, which need to be further explored.
For instance, we do not know the QOT minimizers 
for the type-XX cost function $\min\{|x-x'|,|y-y'|\}$, although the QOT minimizers for similar cost functions $\max\{|x-x'|,|y-y'|\}$ 
and $\min\{x-x',y-y'\}$  are  solved
in Example \ref{ex:th-sub}.
As another example, we do not know the QOT minimizers for the cost function 
$|(x-x')(y-y')|^q,\,1<q\leq 2$
when $\mu,\nu$ are not symmetric (the symmetric case is solved in Theorem \ref{thm:p-cost}).

  \end{enumerate}

\end{enumerate}

\section{Omitted proofs of results from Section \ref{sec:general}}\label{app:sec3}

\begin{proof}[Proof of Proposition \ref{prop:convex?}]
    Consider distinct transport plans $\pi_0,\pi_1\in\Pi(\mu,\nu)$ and denote by $\pi_\lambda=(1-\lambda)\pi_0+\lambda\pi_1$ their convex combination for $\lambda\in[0,1]$.  By symmetry of $\phi$, it holds
    \begin{align}
      \begin{split}
           & (1-\lambda)\iint c\,\d\pi_0\otimes\d\pi_0+\lambda\iint c\,\d\pi_1\otimes\d\pi_1 -\iint c\,\d\pi_\lambda\otimes\d\pi_\lambda\\
        &=\lambda(1-\lambda)\iint c\,\d(\pi_0-\pi_1)\otimes\d(\pi_0-\pi_1).
      \end{split}\label{eq:lambda..}
    \end{align}
    Let $\{\pi^{(n)}\}_{n\in\N}$ be a sequence of signed atomic measures converging weakly to the signed measure $\pi_0-\pi_1$. It follows from \citet[Theorem 2.8]{billingsley2013convergence} that $\pi^{(n)}\otimes\pi^{(n)}\to (\pi_0-\pi_1)\otimes(\pi_0-\pi_1)$ weakly. By the positive definiteness of $\phi$, $\iint c\,\d\pi^{(n)}\otimes\d\pi^{(n)}\geq 0$ for each $n$. Since $c$ is bounded and continuous, 
    $$\iint c\,\d\pi^{(n)}\otimes\d\pi^{(n)}\to \iint c\,\d(\pi_0-\pi_1)\otimes\d(\pi_0-\pi_1)\quad\text{as }n\to\infty.$$
    This shows that \eqref{eq:lambda..} is nonnegative, and hence $\pi\mapsto \iint c\,\d\pi\otimes\d\pi$ is convex.
\end{proof}

\begin{proof}[Proof of Proposition \ref{prop:existence uniqueness}]
     Suppose that $\pi_n\to\pi$ in the weak topology on $\Pi(\mu,\nu)$. Theorem 2.8 of \cite{billingsley2013convergence} yields that $\pi_n\otimes\pi_n\to\pi\otimes\pi$ weakly in the space of probability measures on $(\X\times\Y)^2$. Since $c\in \mathcal C(\mu,\nu)$ is lower semi-continuous, the map
     $$\pi\mapsto \iint c\,\d\pi\otimes\d\pi$$
     is lower semi-continuous by the Portmanteau lemma. 
  Since $\Pi(\mu,\nu)$ is weakly compact, a minimizer of \eqref{eq:QOT cost} exists.   The second claim follows immediately since the set $\T(\mu,\nu)$ of Monge transport maps is weakly dense in $\Pi(\mu,\nu)$ for $\mu$ atomless and $\X$ compact (Theorem 1.32 of \cite{Santambrogio:2015}).
\end{proof}

\begin{proof}[Proof of Proposition \ref{prop:stability}]
A standard argument using Prokhorov's theorem shows that $\{\mu_n\}$ and $\{\nu_n\}$ are equi-tight, and hence $\{\pi_n\}$ is relatively compact; see the proof of Theorem 6.8 of \cite{ambrosio2021lectures}. Let $\pi\in \Pi(\mu,\nu)$ be a limit point of $\{\pi_n\}$. Theorem 2.8 of \cite{billingsley2013convergence} then implies that $\pi_n\otimes\pi_n\to\pi\otimes\pi$ weakly. Since $c$ is continuous and satisfies \eqref{eq:uniform integrable}, we have $\iint c\,\d\pi_n\otimes\d\pi_n\to \iint c\,\d\pi\otimes\d\pi$ (see \citet[Theorem 2.20]{van2000asymptotic} and the example that follows). On the other hand, for any $\hat{\pi}\in\Pi(\mu,\nu)$, Sklar's theorem (\citet[Theorem 7.3]{mcneil2015quantitative}) implies that there exists a copula $C$ such that the cdf of {$\hat{\pi}$} is equal to $C(F_\mu,F_{\nu})$, where $F_{\mu}$ is the cdf of $\mu$. 
Take $\pi_n'$ specified by its cdf $C(F_{\mu_n} ,F_{\nu_n})$.
We have $\pi_n'\in\Pi(\mu_n,\nu_n)$ and $\pi_n'\to\hat{\pi}$ weakly. Therefore, 
$$\iint c\,\d\pi\otimes\d\pi = \lim_{n\to\infty}  \iint c\,\d\pi_n\otimes\d\pi_n \le \lim _{n\to\infty}  \iint c\,\d\pi_n'\otimes\d\pi_n'=\iint c\,\d\hat{\pi}\otimes\d\hat{\pi}.$$
Altogether, we conclude that $\pi$ is a QOT minimizer with marginals $\mu,\nu$ and cost function  $c$.    
\end{proof}

\begin{proof}[Proof of Proposition \ref{prop:lb from conditioning}]
    For each $\pi\in \Pi(\mu,\nu)$, we have by the Fubini--Tonelli theorem,
    \begin{align}
        \begin{split}
            \iint c(x,y,x',y')\,\d \pi(x,y)\,\d \pi (x',y')&=\int\bigg(\int c(x,y,x',y')\,\d \pi (x',y')\bigg)\,\d \pi(x,y) \\
        &\geq \int\C_{c_{x,y}}(\mu,\nu)\,\d \pi(x,y)\geq \C_{\hat{c}}(\mu,\nu).
        \end{split}\label{eq:lb cond}
    \end{align}
    This proves the first claim. The second claim follows by noting that, under the given assumptions, both inequalities in \eqref{eq:lb cond} are equalities.
\end{proof}

\begin{proof}[Proof of Proposition \ref{prop:lb from multimarginal}]
We extend the domain of the infimum by considering the infimum over a larger class of probability measures on $(\X\times\Y)^2$ that contains $\pi\otimes\pi$. 
Define $\Pi_{f,g}$ as the set of probability measures $\tilde{\pi}$ on $(\X\times\Y)^2$ such that for $(X,Y,X',Y')\lawis\tilde{\pi}$, we have $f(X,X')\lawis \mu_f$ and $g(Y,Y')\lawis \nu_g$. Clearly, $\pi\otimes\pi\in\Pi_{f,g}$. This implies that
\begin{align*}
    \begin{split}
        \inf_{\pi\in\Pi(\mu,\nu)}\iint h(f(x,x'),g(y,y'))\,\d \pi(x,y)\,\d \pi (x',y')
        &\geq \inf_{\tilde{\pi}\in \Pi_{f,g}}\int h(f(x,x'),g(y,y'))\,\d\tilde{\pi}(x, y,  x', y')\\
        &=\inf_{\hat{\pi}\in\Pi(\mu_f,\nu_g)}\int h(\xi,\zeta)\,\d\hat{\pi}(\xi,\zeta),
    \end{split}
\end{align*}
as desired.  
\end{proof}

\section{Omitted proofs of results from Section \ref{sec:explicit}}\label{app:sec4}

{We prove Proposition \ref{ex:quadratic cost-f} first. The proof of Proposition \ref{ex:quadratic cost} uses Proposition \ref{ex:quadratic cost-f}.}

\begin{proof}[Proof of Proposition \ref{ex:quadratic cost-f}]
{
 Using the independence of $(X,Y)$ and $(X',Y')$, we see that the objective $\E_{\pi\otimes\pi}[c(X,Y,X',Y')]$ is a linear or quadratic function of $\E_\pi[f(X,Y)]$. On the other hand, since $f$ is submodular, the upper and lower bounds for $\E_\pi[f(X,Y)]$ are attained explicitly  by  $\pi_{\mathrm{com}}$ and $\pi_{\mathrm{ant}}$, respectively. In addition, for any $\beta $ in the interval    \begin{equation}        \label{eq:QC3} \Big[\inf_{\pi\in\Pi(\mu,\nu)}\E_\pi[f(X,Y)],\sup_{\pi\in\Pi(\mu,\nu)}\E_\pi[f(X,Y)]\Big],    \end{equation} 
    there exists $\lambda \in [0,1]$ such that  $\E_{\pi_{\rm x}^\lambda }[f(X,Y)]=\beta$ because  $\lambda \mapsto \E_{\pi_{\rm x}^\lambda}[f(X,Y)]$ is affine.  This shows that a QOT minimizer $\pi_{\lambda}$ exists. 
    The last statement follows by noting that, for some choices of $C_1,C_2,\alpha_1,\alpha_2\in\R$ (which are arbitrary) in \eqref{eq:QC-f}, the range \eqref{eq:QC3} 
 for $\beta$ is not a singleton and   
    any $\beta$ in \eqref{eq:QC3}  can be a unique minimizer. 
Hence, any $\lambda\in [0,1]$ can yield a unique minimizer
     in the class $(\pi_{\rm x}^\lambda)_{\lambda \in [0,1]}$.   
   }
    \end{proof}

\begin{proof}[Proof of Proposition \ref{ex:quadratic cost}]
 {
 By independence, we may write 
    $$\E_{\pi\otimes\pi}[c(X,Y,X',Y')]=\E_\pi[f(X,Y)]\E_\pi[g(X,Y)],\quad \pi\in\Pi(\mu,\nu).$$
    The marginal terms of $f,g$ have constant expectations, so there exist constants $C_1,C_2,\alpha_1,\alpha_2\in\R$ such that
    \begin{equation}        \label{eq:QC2}\E_\pi[f(X,Y)]\E_\pi[g(X,Y)]=(C_1+\alpha_1\E_\pi[XY])(C_2+\alpha_2\E_\pi[XY]). 
    \end{equation} 
   The desired result is then a special case of Proposition \ref{ex:quadratic cost-f} since $(x,y)\mapsto xy$ is supermodular. The last statement follows since any combination of $(C_1,C_2,\alpha_1,\alpha_2)$ can be achieved by certain choices of the quadratic functions $f,g$.
     }
\end{proof}

\begin{proof}[Proof of Theorem \ref{prop:comonotone}]
Define the function $c_{x,y}(x',y')=c(x,y,x',y')$ and recall the notation in Proposition \ref{prop:lb from conditioning}.
 It follows from the optimality of $\pi_{\rm com}$ for submodular cost functions and the Fubini--Tonelli theorem that 
$$ \iint c\, \d \pi_{\rm com}\otimes \d \pi_{\rm com} =\int\bigg(\int c(x,y,x',y')\,\d \pi_{\mathrm{com}} (x',y')\bigg)\,\d \pi_{\mathrm{com}}(x,y) = \int\C_{c_{x,y}}(\mu,\nu)\,\d \pi_{\mathrm{com}}(x,y).$$
Note that the function $(x,y)\mapsto \C_{c_{x,y}}(\mu,\nu)$ is also submodular, as a weighted combination of submodular functions. As a consequence,
$$\int\C_{c_{x,y}}(\mu,\nu)\,\d \pi_{\mathrm{com}}(x,y)= \C_{\hat{c}}(\mu,\nu).$$
In other words, both inequalities in \eqref{eq:lb cond} are equalities for $\pi=\pi_{\mathrm{com}}$. This implies that $\pi_{\mathrm{com}}$ must be a minimizer because of Proposition \ref{prop:lb from conditioning}.   The supermodular case is analogous. 
\end{proof}

\begin{proof}[Proof of Theorem \ref{prop:cocounter}]
Without loss of generality, we can assume $\mu=\nu$, as the location-scale transform can be absorbed into $h$ without affecting submodularity or the uniqueness. 
 Let   $\kappa$ be the law of $|X-X'|$, where $X\lawis\mu$ and $X'$ is an independent copy of $X$.   
Let $\widehat{\Pi}$ be the set of probability measures $\hat{\pi}$ on $\R^4$ such that for $(X,Y,X',Y')\lawis\hat{\pi}$, we have $|X-X'|,|Y-Y'|\lawis \kappa$. Proposition \ref{prop:lb from multimarginal} then implies 
\begin{align} 
        \inf_{\pi\in \Pi(\mu,\nu)}\iint  c\, \d \pi \otimes \d \pi  \geq \inf_{\hat{\pi}\in \widehat{\Pi}}\int_{\R^4}h(|x-x'|,|y-y'|)\,\d\hat{\pi}(x,y,x',y').
    \label{eq:pi}
\end{align}
The integral on the right-hand side of \eqref{eq:pi} depends only on the coupling of $(|Y-Y'|,|X-X'|)$ under the law $\hat{\pi}$. Since the marginals of $|X-X'|$ and $|Y-Y'|$ both follow the law $\kappa$ under any $\hat{\pi}\in \widehat{\Pi}$, the right-hand side of \eqref{eq:pi} coincides with the optimal transport cost between laws $\kappa$ and $\kappa$ with cost function $h$. If $h$ is submodular, the problem is uniquely minimized by $\pi_{\mathrm{com}}$. This is equivalent to $|X-X'|=|Y-Y'|$ almost surely.
To show that the comonotone coupling is a minimizer, observe that under $\pi_{\mathrm{com}}$, $X=Y$ and $X'=Y'$ hold, and hence $|X-X'|=|Y-Y'|$.

Assume that $h$ is strictly submodular. The right-hand side of \eqref{eq:pi} is then uniquely minimized by the comonotone coupling, or $|X-X'|=|Y-Y'|$ almost surely. It remains to show that $\pi_{\mathrm{com}}$ (and $\pi_{\mathrm{ant}}$ if $\mu$ is symmetric) is the unique transport plan that verifies $|X-X'|=|Y-Y'|$.  
Indeed, this relation implies  
$$(X-X’+Y-Y’)(X-X’-Y+Y’) = (X-X’)^2 - (Y-Y’)^2=0.$$
 Hence, either $X+Y=X’+Y’$ or $X-Y=X’-Y’$ almost surely. Since the two sides are independent, we have either $X+Y$ is a constant (only if $(X,Y)\lawis\pi_{\mathrm{ant}}$) or $X-Y$ is a constant (only if $(X,Y)\lawis\pi_{\mathrm{com}}$). Since $\mu=\nu$, the comonotone coupling verifies $X-Y=0$; the antimonotone coupling verifies $X+Y$ is a constant if and only if $\mu$ is symmetric. This completes the proof. 
\end{proof}

To prove Theorem \ref{prop:wedge}, we need the following lemma. 

\begin{lemma}\label{lemma:opt cts}
Fix $\alpha,\beta\in(0,1]$ and $\gamma>0$. Consider the following optimization problem:
    \begin{align*}
        \mbox{maximize}\quad&\p(|X-Y|\leq\gamma)\\
        \mbox{subject to}\quad &X,Y\mbox{ are independent with respective densities }f_X,f_Y;\\
        &f_X(x)\leq \alpha^{-1}\bone_{[0,1]}(x);\\
        &f_Y(y)\leq \beta^{-1}\bone_{[0,1]}(y).
    \end{align*}
    Then an optimizer $(X,Y)$ is given by 
    $X\lawis\mathrm{U}((1-\alpha)/2,(1+\alpha)/2)$ and $Y\lawis\mathrm{U}((1-\beta)/2,(1+\beta)/2)$.    
\end{lemma}
The proof of Lemma \ref{lemma:opt cts} is based on a result in \cite{burkard1998quadratic} on discrete assignment. The following result is equivalent to Lemma 2.8 of \cite{burkard1998quadratic}, which is the discrete version of Lemma \ref{lemma:opt cts}.

\begin{lemma}\label{lemma:opt discrete}
Let $p,q,n$ be integers satisfying $1\leq p,q\leq n$. For a fixed $\gamma>0$, consider the following optimization problem:
    \begin{align*}
        \mbox{maximize}\quad&\p(|X-Y|\leq\gamma)\\
        \mbox{subject to}\quad &X,Y\mbox{ are independent};\\
        &X\mbox{ is uniformly distributed on }p \mbox{ points in }[n];\\
        &Y\mbox{ is uniformly distributed on }q \mbox{ points in }[n].
    \end{align*}
    Then an optimizer $(X,Y)$ is given by $X$ being uniformly distributed on the last $p$ points of the finite sequence
    \begin{align}
        1,\,n,\,2,\,n-1,\,3,\dots,\label{eq:seq}
    \end{align}
    and $Y$ being uniformly distributed on the last $q$ points of \eqref{eq:seq}.
\end{lemma}


\begin{proof}[Proof of Lemma \ref{lemma:opt cts}] Let $(X,Y)$ be given by the claimed optimal solution. 
Suppose on the contrary that there exist independent random variables $\hat{X},\hat{Y}$ satisfying the constraints, whose joint law $(\hat{X},\hat{Y})$ is different from $(X,Y)$, and furthermore,
\begin{align}
    \p(|\hat{X}-\hat{Y}|\leq\gamma)>\p(|X-Y|\leq\gamma)\label{eq:hats}
\end{align}for some $\gamma>0$. 
Note that $(\hat{X},\hat{Y})$ is absolutely continuous with bounded density by the constraints. Then there exist a sequence of random variables $(\hat{X}_n,\hat{Y}_n)_{n\geq 1}$ such that:
\begin{itemize}
    \item for each $n$, $\hat{X}_n$ and $\hat{Y}_n$ are independent;
    \item $\hat{X}_n$ (resp.~$\hat{Y}_n$) is supported uniformly on at most $\lfloor \alpha n\rfloor$ (resp.~$\lfloor \beta n\rfloor$) points of $\mathbb{Z}/n\cap [0,1]$;
    \item $(\hat{X}_n,\hat{Y}_n)\to (\hat{X},\hat{Y})$ in distribution.
\end{itemize}
Similarly, there exist 
a sequence of random variables $({X}_n,{Y}_n)_{n\geq 1}$ such that:
\begin{itemize}
    \item for each $n$, ${X}_n$ and ${Y}_n$ are independent;
    \item $X_n$ (resp.~$Y_n$) is uniformly supported on the last $\lfloor \alpha n\rfloor$ (resp.~$\lfloor \beta n\rfloor$) elements of \eqref{eq:seq} scaled by $1/n$.
\end{itemize}
It follows that $(X_n,Y_n)\to (X,Y)$ in distribution.
By Lemma \ref{lemma:opt discrete}, we have for each $n$ that
$$\p(|\hat{X}_n-\hat{Y}_n|\leq\gamma)\leq \p(|X_n-Y_n|\leq\gamma).$$
Taking the limit in $n$ and applying the Portmanteau lemma, we have
$$\p(|\hat{X}-\hat{Y}|\leq\gamma)\leq \p(|X-Y|\leq\gamma),$$
leading to a contradiction against \eqref{eq:hats}.   
\end{proof}

\begin{proof}[Proof of Theorem \ref{prop:wedge}] By absorbing $a$ and $b$ into $f$, without loss of generality we can assume $a=0$ and $b=1$.
    We first apply the decomposition
    $$f(|x-x'|)=\int \bone_{\{|x-x'|\geq u\}}\d\lambda(u)=:\int c_u(x,x')\,\d\lambda(u),$$
    and 
   $$g(y,y')=\int \bone_{[v,\infty)\times [v',\infty)}(y,y')\,\d \eta(v,v')=:\int c_{v,v'}(y,y')\,\d \eta(v,v')$$ 
    where $\lambda,\eta$ are some positive measures. For instance, the measure $\eta$ may be defined via
    $$\eta((s,t]\times (s',t'])=g(t,t')+g(s,s')-g(s,t')-g(t,s')\quad\text{ for }\quad s<t,~s'<t'.$$
    By the monotone convergence and Fubini theorems, it remains to show that for every fixed $u,u',v$, $\pi_{\rm v}$ is a minimizer of the QOT problem with the cost function
    $$c(x,y,x',y')=c_{u}(x,x')c_{v,v'}(y,y')=\bone_{\{|x-x'|\geq u\}}\bone_{[v,\infty)\times [v',\infty)}(y,y').$$
    Such a problem is equivalent to finding $(X,Y,X',Y')\lawis \pi$ that minimizes
    \begin{align}
        \p(|X-X'|\geq u,\,Y\geq v,\,Y'\geq v'),\label{eq:XX'2}
    \end{align}
    subject to $(X,Y)\dd (X',Y')$, the independence of $(X,Y)$ and $ (X',Y')$, and the marginal constraints from $\pi$ that $X\lawis\mu$ and $Y\lawis\nu$. We focus on the case where $\p(Y\geq v)>0$ and $\p(Y'\geq v')>0$, otherwise the problem is trivial as \eqref{eq:XX'2} evaluates to zero. Without loss of generality, we may first remove the constraint that $(X,Y)\dd (X',Y')$ and later show that it is indeed satisfied by the minimizer. Denote by $\xi_1$ the law of $X\mid Y\geq v$ and $\xi_2$ the law of $X'\mid Y'\geq v'$. Minimizing \eqref{eq:XX'2} is then equivalent to minimizing $\p(|\xi_1-\xi_2|\geq u)$, where $\xi_1,\xi_2$ are independent. Observe that the marginal constraints on $\pi$ are equivalent to constraining $\xi_1$ having density bounded by $1/\p(Y\geq v)$ on $[0,1]$, and similarly $\xi_2$ having density bounded by $1/\p(Y'\geq v')$ on $[0,1]$. Indeed, any such law $\xi_1$ can be written as the law of $X\mid Y\geq v$ for some coupling $(X,Y)$ satisfying the marginal constraints. In other words, we have reduced to the following problem:
    \begin{align*}
        \mbox{to minimize}\quad&\p(|\xi_1-\xi_2|\geq u)\\
        \mbox{subject to}\quad &\xi_1,\xi_2\mbox{ are independent r.v.s on }[0,1]\mbox{ with respective densities }f_{\xi_1},f_{\xi_2};\\
        &f_{\xi_1}\leq 1/\p(Y\geq v)\mbox{ on }[0,1];\\
        &f_{\xi_2}\leq 1/\p(Y'\geq v')\mbox{ on }[0,1].
    \end{align*}
    By Lemma \ref{lemma:opt cts}, a solution is given by 
    $$\xi_1\lawis \mathrm{U}\Big(\frac{1-\p(Y\geq v)}{2},\frac{1+\p(Y\geq v)}{2}\Big)\quad\text{ and }\quad \xi_2\lawis \mathrm{U}\Big(\frac{1-\p(Y'\geq v')}{2},\frac{1+\p(Y'\geq v')}{2}\Big).$$
    By Definition \ref{def:wedge}, the V-transport satisfies that for each $v\in\R$, 
    $$X\mid Y\geq v\lawis \mathrm{U}\Big(\frac{1-\p(Y\geq v)}{2},\frac{1+\p(Y\geq v)}{2}\Big).$$
    Since $\pi_{\rm v}$ does not depend on the choices of $u,u',v$, the constraint $(X,Y)\dd (X',Y')$ in the minimization problem \eqref{eq:XX'2} is automatically satisfied. This completes the proof.    
\end{proof}

\section{Omitted proofs of results from Section \ref{sec:diamond}}\label{app:sec5}

\begin{proof}[Proof of Lemma \ref{lemma:1/2}]
By \eqref{eq:4c}, it remains to show that $\tilde{c}$ is supermodular on the support of $\pi_{\rm ind}$ in $(-\infty,0]^2$, i.e., with $Q_\mu,\,Q_\nu$ denoting the left quantile functions of $\mu,\nu$, 
\begin{align}
    \int c_{xy}(Q_\mu(p),Q_\nu(q),Q_\mu(p'),Q_\nu(q'))\,\d C_{\mathrm{dia}}(p',q')\geq 0,\quad p,q\in(0,1/2),\label{eq:superm}
\end{align}
where $c_{xy}$ denotes the second-order partial derivative of $c$ with respect to the first two variables. 
Let $\psi:\R_+\to\R$ be the function given by $\psi(u)=\phi(u^2)$.
For notational simplicity, let $g:[0,1]^4\to \R$ be given by 
$$
g(p,q,p',q'):= c_{xy}(Q_\mu(p),Q_\nu(q),Q_\mu(p'),Q_\nu(q'))
$$
and $\eta_\mu,\eta_\nu:\{(p,p'):0\leq p'\leq p\leq 1\}\to \R$ be given by 
$$
\eta_\mu(p,p'):= \psi'(Q_\mu(p)-Q_\mu(p'))\quad\text{ and }\quad \eta_\nu(p,p'):= \psi'(Q_\nu(p)-Q_\nu(p')).
$$
Using the assumption 
$$c(x,y,x',y')=\phi((x-x')^2)\phi((y-y')^2)=\psi(|x-x'|)\psi(|y-y'|)$$
and monotonicity of $Q_\mu,\,Q_\nu$, we have (in the a.e.~sense)
\begin{align}
g(p,q,p',q')=\sgn(p-p')\,\sgn(q-q') \psi'(|Q_\mu(p)-Q_\mu(p')|) \psi'(|Q_\mu(q)-Q_\mu(q')|) .\label{eq:cxy}
\end{align}

We first deal with the case $p+q\leq 1/2$. Using the definition of $\pi_{\mathrm{dia}}$, we compute the left-hand side of \eqref{eq:superm} as
{\allowdisplaybreaks\begin{align*}
  & \int g(p,q,p',q')\,\d C_{\mathrm{dia}}(p',q')\\
    &=\int_0^{1/2}g\left(p,q,p',\frac{1}{2}+p'\right)\d p'+\int_0^{1/2}g\left(p,q,p',\frac{1}{2}-p'\right)\d p'\\
    &\hspace{1cm}+\int^1_{1/2}g\left(p,q,p',\frac{3}{2}-p'\right)\d p'+\int^1_{1/2}g\left(p,q,p',p'-\frac{1}{2}\right)\d p'\\
    &\geq -\int_0^p \eta_\mu(p,p')\eta_\nu\left(\frac{1}{2}+p',q\right)\d p'+\int_p^{1/2}\eta_\nu(p',p)\eta_\nu\left(\frac{1}{2}+p',q\right)\d p'\\
   &\hspace{0.3cm}-\int_0^{1/2}\psi'(|Q_\mu(p)-Q_\mu(p')|)\psi'\left(\left|Q_\nu\left(\frac{1}{2}-p'\right)-Q_\nu(q)\right|\right)\d p'+\int_{1/2}^1\eta_\nu(p',p)\eta_\nu\left(\frac{3}{2}-p',q\right)\d p'\\
    &\hspace{0.3cm}-\int_{1/2}^{1/2+q}\eta_\nu(p',p)\eta_\nu\left(q,p'-\frac{1}{2}\right)\d p'+\int_{1/2+q}^1\eta_\nu(p',p)\eta_\nu\left(p'-\frac{1}{2},q\right)\d p'\\
        &=\underbrace{\int_{1-p}^1\eta_\nu(p',p)\eta_\nu\left(\frac{3}{2}-p',q\right)\d p'-\int_0^p \eta_\mu(p,p')\eta_\nu\left(\frac{1}{2}+p',q\right)\d p'}_{\displaystyle=:I_1}\\
    &\hspace{1cm}+\underbrace{\int_{1-p}^1\eta_\nu(p',p)\eta_\nu\left(p'-\frac{1}{2},q\right)\d p'-\int_0^{p}\eta_\mu(p,p')\eta_\nu\left(\frac{1}{2}-p',q\right)\d p'}_{\displaystyle=:I_2}\\
     &\hspace{1cm}+\underbrace{\int_{1/2+q}^{1-p}\eta_\nu(p',p)\eta_\nu\left(p'-\frac{1}{2},q\right)\d p'-\int_p^{1/2-q}\eta_\nu(p',p)\eta_\nu\left(\frac{1}{2}-p',q\right)\d p'}_{\displaystyle=:I_3}\\
         &\hspace{1cm}+\underbrace{\int_p^{1/2}\eta_\nu(p',p)\eta_\nu\left(\frac{1}{2}+p',q\right)\d p'-\int_{1/2-q}^{1/2} \eta_\nu(p',p)\eta_\nu\left(q,\frac{1}{2}-p'\right)\d p'}_{\displaystyle=:I_4}\\
    &\hspace{1cm}+\underbrace{\int_{1/2}^{1-p}\eta_\nu(p',p)\eta_\nu\left(\frac{3}{2}-p',q\right)\d p'-\int_{1/2}^{1/2+q}\eta_\nu(p',p)\eta_\nu\left(q,p'-\frac{1}{2}\right)\d p'}_{\displaystyle=:I_5}. 
\end{align*}
 By \eqref{eq:phi condition}, we have for all {$u$ such that $u^2\in \mathcal{D}$}, }
\begin{align}
    \psi''(u)=2\phi'(u^2)+4u^2\phi''(u^2)\leq 0.\label{eq:psi''}
\end{align}As a consequence, $\psi'$ is decreasing on the domain of interest. Since $\phi'\leq 0$, we have $\psi'\leq 0$.
Using a change of variable, we obtain
$$I_1=\int_0^p (\psi'(Q_\mu(1-p')-Q_\mu(p))-\eta_\mu(p,p'))\eta_\nu\left(\frac{1}{2}+p',q\right)\d p'\geq 0.$$
A similar argument using the symmetry properties of $Q_\mu$ and $Q_\nu$ shows that $I_j\geq 0$ for each $j=2,3,4,5$. Combining the above yields 
$$\int g(p,q,p',q')\,\d C_{\mathrm{dia}}(p',q')\geq 0,$$
proving \eqref{eq:superm} in the case $p+q\leq 1/2$.

Next, we deal with the case $1/2\leq p+q\leq 1$. It remains to show that 
\begin{align}
    \int g\left(\frac{1}{2},q,p',q'\right)\d C_{\mathrm{dia}}(p',q')=0,\quad \text{for all }q\in[0,1/2]\label{eq:toshow c1}
\end{align}
and that for all $ p,q\in[0,1/2]$,
\begin{align}
    \begin{split}
        &\frac{\partial}{\partial x}\int g(p,q,p',q')\,\d C_{\mathrm{dia}}(p',q')=\int c_{xxy}(Q_\mu(p),Q_\nu(q),Q_\mu(p'),Q_\nu(q'))\,\d C_{\mathrm{dia}}(p',q')\leq 0.
    \end{split}\label{eq:toshow c2}
\end{align}
Indeed, integrating \eqref{eq:toshow c2} and using \eqref{eq:toshow c1} imply \eqref{eq:superm}. By \eqref{eq:4c} and the smoothness of $\tilde{c}$, we have $\tilde{c}_{x}(Q_\mu(1/2),Q_\nu(q))=0$ for each $q\in[0,1]$, and hence \eqref{eq:toshow c1} follows. To prove \eqref{eq:toshow c2}, we first differentiate \eqref{eq:cxy} to get that in the a.e.~sense,
\begin{align}
    \begin{split}
        &c_{xxy}(Q_\mu(p),Q_\nu(q),Q_\mu(p'),Q_\nu(q'))\\
        &=\sgn(p-p')^2\sgn(q-q')\psi''(|Q_\mu(p)-Q_\mu(p')|)\psi'(|Q_\nu(q)-Q_\nu(q')|)\\
    &\hspace{4cm}+2\delta_{p-p'}\sgn(q-q')\psi'(|Q_\mu(p)-Q_\mu(p')|)\psi'(|Q_\nu(q)-Q_\nu(q')|)\\
    &=\sgn(q-q')\psi''(|Q_\mu(p)-Q_\mu(p')|)\psi'(|Q_\nu(q)-Q_\nu(q')|)\\
    &\hspace{4cm}+2\psi'(0)\,\delta_{p-p'}\sgn(q-q')\psi'(|Q_\nu(q)-Q_\nu(q')|).
    \end{split}\label{eq:cxxy}
\end{align}
Let $p,q\in[0,1/2]$ with $p+q\geq 1/2$. We first check that
\begin{align}
    \int \sgn(q-q')\psi''(|Q_\mu(p)-Q_\mu(p')|)\psi'(|Q_\nu(q)-Q_\nu(q')|)\,\d C_{\mathrm{dia}}(p',q')\leq 0.\label{eq:toshow c3}
\end{align}
Recall that $\psi'\leq 0$ and $\psi''\leq 0$ by \eqref{eq:psi''}. For notational simplicity, let 
$\hat{\eta}:\{(p,p'):0\leq p'\leq p\leq 1\}\to \R$ be given by 
$$
\hat{\eta}(p,p'):= \psi''(|Q_\mu(p)-Q_\mu(p')|).
$$
We compute
{\allowdisplaybreaks
\begin{align*}
    &\int \sgn(q-q')\hat{\eta}(p,p')\psi'(|Q_\nu(q)-Q_\nu(q')|)\,\d C_{\mathrm{dia}}(p',q')\\
    &=\int_0^{1/2}\sgn\left(q-\frac{1}{2}-p'\right)\hat{\eta}(p,p')\psi'\left(\left|Q_\nu(y)-Q_\nu\left(\frac{1}{2}+p'\right)\right|\right)\d p'\\
    &\hspace{1cm}+\int_0^{1/2}\sgn\left(q-\frac{1}{2}+p'\right)\hat{\eta}(p,p')\psi'\left(\left|Q_\nu(y)-Q_\nu\left(\frac{1}{2}-p'\right)\right|\right)\d p'\\
    &\hspace{1cm}+\int_{1/2}^1\sgn\left(q-\frac{3}{2}+p'\right)\hat{\eta}(p,p')\psi'\left(\left|Q_\nu(y)-Q_\nu\left(\frac{3}{2}-p'\right)\right|\right)\d p'\\
    &\hspace{1cm}+\int_{1/2}^1\sgn\left(q+\frac{1}{2}-p'\right)\hat{\eta}(p,p')\psi'\left(\left|Q_\nu(y)-Q_\nu\left(p'-\frac{1}{2}\right)\right|\right)\d p'\\
    &=-\int_0^{p}\hat{\eta}(p,p')\eta_\nu\left(\frac{1}{2}+p',q\right)\d p'-\int_p^{1/2}\hat{\eta}(p,p')\eta_\nu\left(\frac{1}{2}+p',q\right)\d p'\\
    &\hspace{1cm}-\int_0^{1/2-q}\hat{\eta}(p,p')\eta_\nu\left(q,\frac{1}{2}-p'\right)\d p'+\int_{1/2-q}^{1/2}\hat{\eta}(p,p')\eta_\nu\left(q,\frac{1}{2}-p'\right)\d p'\\
    &\hspace{1cm}-\int_{1/2}^{1/2+q}\hat{\eta}(p,p')\eta_\nu\left(\frac{3}{2}-p',q\right)\d p'-\int_{1/2+q}^1\hat{\eta}(p,p')\eta_\nu\left(\frac{3}{2}-p',q\right)\d p'\\
    &\hspace{1cm}+\int_{1/2}^{1/2+q}\hat{\eta}(p,p')\eta_\nu\left(q,p'-\frac{1}{2}\right)\d p'-\int_{1/2+q}^{1}\hat{\eta}(p,p')\eta_\nu\left(p'-\frac{1}{2},q\right)\d p'\\
    &\leq \int_{1/2}^{1/2+q}\hat{\eta}(p,p')\eta_\nu\left(q,p'-\frac{1}{2}\right)\d p'-\int_{1/2}^{1/2+q}\hat{\eta}(p,p')\eta_\nu\left(\frac{3}{2}-p',q\right)\d p'\\
    &\hspace{1cm}+\int_{1/2-q}^{p}\hat{\eta}(p,p')\eta_\nu\left(q,\frac{1}{2}-p'\right)\d p'-\int_{1/2-q}^{p}\hat{\eta}(p,p')\eta_\nu\left(\frac{1}{2}+p',q\right)\d p'\\
     &\hspace{1cm}+\int_{p}^{1/2}\hat{\eta}(p,p')\eta_\nu\left(q,\frac{1}{2}-p'\right)\d p'-\int_{p}^{1/2}\hat{\eta}(p,p')\eta_\nu\left(\frac{1}{2}+p',q\right)\d p'\\
     &\leq 0,
\end{align*}
  where the last step follows from $\hat{\eta}(p,p')\leq 0$ along with the following considerations:}
\begin{itemize}
    \item since $q\leq 1/2$, it holds $2Q_\nu(q)\leq 0= Q_\nu(p'-\frac{1}{2})+Q_\nu(\frac{3}{2}-p')$, so that $\eta_\nu(q,p'-\frac{1}{2})\geq \eta_\nu(\frac{3}{2}-p',q)$;
    \item again since $q\leq 1/2$, we have $2Q_\nu(q)\leq 0=Q_\nu(\frac{1}{2}-p')+ Q_\nu(\frac{1}{2}+p')$, so $\eta_\nu(q,\frac{1}{2}-p')\geq \eta_\nu(\frac{1}{2}+p',q)$.
\end{itemize}
This proves \eqref{eq:toshow c3}.

In addition, using $p,q\in[0,1/2]$, \eqref{eq:psi''}, and the definition of $ C_{\mathrm{dia}}$, we obtain
\begin{align*}
    &\int 2\delta_{p-p'}\sgn(q-q')\psi'(|Q_\nu(q)-Q_\nu(q')|)\,\d C_{\mathrm{dia}}(p',q')\\
    &=\sgn\left(q-\frac{1}{2}-p\right)\psi'\left(\left|Q_\nu(q)-Q_\nu\left(\frac{1}{2}+p\right)\right|\right)+\sgn(q-\frac{1}{2}+p)\psi'\left(\left|Q_\nu(q)-Q_\nu\left(\frac{1}{2}-p\right)\right|\right)\\
    &=\psi'\left(Q_\nu(q)-Q_\nu\left(\frac{1}{2}-p\right)\right)-\psi'\left(Q_\nu\left(p+\frac{1}{2}\right)-Q_\nu(q)\right) \geq 0.
\end{align*}
Therefore,
\begin{align}
    \int 2\psi'(0)\,\delta_{p-p'}\sgn(q-q')\psi'(|Q_\nu(q)-Q_\nu(q')|)\,\d C_{\mathrm{dia}}(p',q')\leq 0.\label{eq:toshow c4}
\end{align}
Combining \eqref{eq:cxxy}, \eqref{eq:toshow c3}, and \eqref{eq:toshow c4} yields \eqref{eq:toshow c2} and thus proves \eqref{eq:superm} in the case $p+q\geq 1/2$.
\end{proof}

\end{document}